\documentclass{article}
\usepackage{graphicx}

\usepackage[fleqn]{amsmath}
\usepackage{amsthm}
\usepackage{amsmath}
\usepackage{amssymb}
\usepackage{amsbsy}
\usepackage{amsfonts}
\usepackage{mathrsfs}
\usepackage[all]{xy}
\usepackage{amstext}
\usepackage{amscd}
\usepackage[dvips]{epsfig}
\usepackage{psfrag}
\usepackage{enumerate}
\usepackage{flafter}
\allowdisplaybreaks

\theoremstyle{plain}
\newtheorem{thm}{Theorem}[section]
\newtheorem{prop}[thm]{Proposition}

\newtheorem{lem}[thm]{Lemma}
\theoremstyle{definition}
\newtheorem{exa}[thm]{Example}

\newtheorem{rem}[thm]{Remark}
\newtheorem{defn}[thm]{Definition}

\def\det{\mathop{\mathrm{det}}\nolimits}

\def\Ker{\mathop{\mathrm{Ker}}\nolimits}

\newcommand{\lra}{\longrightarrow}
\newcommand{\ra}{\rightarrow}
\newcommand{\Q}{{\Bbb Q}}

\newcommand{\Z}{{\Bbb Z}}

\newcommand{\pc}[2]{\mbox{$\begin{array}{c}
\includegraphics[scale=#2]{#1.eps}
\end{array}$}}

\begin{document}

\large
\begin{center}
{\bf\Large Twisted Alexander invariants of knot group representations}
\end{center}
\vskip 1.5pc
\begin{center}{Takefumi Nosaka\footnote{
E-mail address: {\tt nosaka@math.titech.ac.jp}
}}\end{center}
\vskip 1pc

\begin{abstract}\baselineskip=12pt \noindent
Given a homomorphism from a knot group to a fixed group, we introduce an element of a $K_1$-group, which is a generalization of (twisted) Alexander polynomials. We compare the $K_1$-class with other Alexander polynomials. In terms of semi-local rings, we compute the $K_1$-classes of some knots and show their non-triviality. We also introduce metabelian  Alexander polynomials.
\end{abstract}
\begin{center}
\normalsize
\baselineskip=17pt
{\bf Keywords} \\
\ \ \ knot, Alexander polynomial, $K_1$-group, Dieudonn\'{e} determinant, Novikov ring, semi-local ring \ \
\end{center}

\large
\baselineskip=16pt
\section{Introduction}
\label{IntroS}
The Alexander polynomial of a knot $K$ in the 3-sphere $S^3$ is a landmark topic in 3-dimensional topology that has provided various topological applications; see, e.g., \cite{Lic}. The classical polynomial is, roughly speaking, a homological study of the abelianization of the fundamental group $\pi_1(S^3 \setminus K) $. Since the beginning of the 21st century, the following generalizations of the Alexander polynomial have been suggested from noncommutative viewpoints: first, the twisted Alexander polynomial \cite{Lin,Wada} can be defined from a representation $ \pi_1(S^3 \setminus K) \ra GL_n(R)$ over a commutative ring $R$. This definition has also provided applications in topology.  Since the twisted polynomial is defined as a determinant, it can be calculated using the Fox derivative; see \cite{Wada,FV}. Second, Cochran \cite{C} and Harvey \cite{Har} define $\Delta_h$ from a locally indicable and amenable group; $\Delta_h$ is called the higher-order Alexander polynomial and lies in a non-commutative ring up to some indeterminacy. However, the coefficients of the polynomial lie in a skew fractional field, making it difficult to quantitatively investigate the polynomial. There are few instances in which the polynomial $\Delta_h$ with non-triviality  has been computed; see, e.g., \cite{Har,GS,Horr}. In contrast, there are a number of studies on the degree of $\Delta_h$; see, e.g., \cite{C,Har,Har3,Horr,Kitayama}.

In this paper, starting from any group $G$ and any homomorphism $\rho: \pi_1(S^3 \setminus K) \ra G $ with a choice of a meridian $\mathfrak{m} \in \pi_1(S^3 \setminus K)$, we introduce a $K_1$-class, $\Delta_{\rho}^{K_1}$, which lies in a quotient group $\mathcal{Q}_{\mathcal{A}, \kappa}$ of a $K_1$-group; see Example \ref{exppp0} and Definition \ref{def11}. The $K$-groups have provided uniform understanding of several mathematical phenomena; see, e.g., \cite{Wei}. For example, some Reidemeister torsions and $L^2$-torsions can be summarized as elements of some $K_1$-groups; see \cite{FL,Kie,Mil,Tur}; the paper \cite{GN} defines some invariants of fibered spaces, which lie in $K_1$-groups of some Novikov rings. 
Similarly, we will see (as in Example \ref{exa1124}) that the above twisted polynomials are recovered from the $K_1$-class. A relation to the higher order polynomial will be studied in a forthcoming paper \cite{Nos}.

In general, it is difficult to compute a $K_1$-group and show the non-triviality of an element in the $K_1$-group. Fortunately, the quotient group $\mathcal{Q}_{\mathcal{A}, \kappa}$ has been previously studied in \cite{PR,P}. Using this study as our motivation, we suggest two maps from $\mathcal{Q}_{\mathcal{A}, \kappa}$ in Section \ref{QSS}: the first is due to a logarithm developed by Pajitnov \cite{P}, and the second is a composite map of a ring homomorphism $\Upsilon$ and Morita equivalence. In Section \ref{ExaSS}, we compute the $K_1$-classes $\Delta_{\rho}^{K_1}$'s of some knots and show their non-triviality using the two maps. As a corollary of this study, we suggest two criterions for detecting fiberedness of knots (see Theorems \ref{thm119} and \ref{ww2}).

In Section \ref{invdefS2}, we consider the case $G \cong H \rtimes \Z$ for some finite group $H$ and show that the pushforward of $\Delta_{\rho}^{K_1} $ by $\Upsilon $ is independent of the choice of $ \mathfrak{m}$ (see Proposition \ref{ww3}). For a special metabelian case, we define a {\it metabelian Alexander polynomial}, $\Delta_{ \rho^{\rm meta}_{ N} }^{K_1}$, as a knot invariant (Definition \ref{def133}), which is a generalization of the twisted Alexander polynomial in \cite{HKL}. The coefficients of $\Delta_{ \rho^{\rm meta}_{ N} }^{K_1}$ appear in commutative polynomials, enabling us to obtain a computable knot invariant. We will investigate properties of $\Delta_{\rho}^{K_1}$ and the metabelian polynomial in a forthcoming paper \cite{Nos}.

\

\noindent
{\textbf{Conventional notation.} For a group $G$ and a commutative ring $A$, we denote the group ring by $A [G]$ over $A$ and the abelianization by $G_{\rm ab}$. Every non-commutative ring $R$ always has 1; $R^{\times}$ refers to the multiplicative group consisting of units in $R$, and $\Sigma_{g,1}$ denotes the connected oriented compact surface of genus $g$ with circle boundary.

\subsection{Acknowledgments}
The author expresses his gratitude to Professors J. Scott Carter, Hiroshi Goda, Takahiro Kitayama, and Andrei Pajitnov for valuable discussions and comments.
He is also grateful to Professors 
Dawid Kielak and Andrew Nicas for telling him some references. 

\section{Review: regular Seifert surfaces and regular Heegaard splitting }
\label{ganeaS}
We start by reviewing terminology of regular Seifert surfaces and spines in \cite{Lin}. We suppose basic knowledge of knot theory as in \cite[\S\S 1--10]{Lic}. A {\it knot} in $S^3$ is a tame embedding of $S^1$ into $S^3$. Let $K$ be an oriented knot in $S^3$. A Seifert surface of $K$ is a surface $F$ embedded in $S^3$ such that $\partial F= K$ and $F$ is homeomorphic to $\Sigma_{g,1}$ for some $g \in \Z$. A {\it spine} of such an $F$ is a bouquet of circles $W \subset F$ such that it is a deformation retract of $F$. A Seifert surface $F$ is said to be {\it regular} if it has a spine $W$ whose embedding in $S^3$ induced by $F \subset S^3$ is isotopic to the standard embedding $W \subset F$. In other words, the regularity refers to a choice of generator $u_1,\dots, u_{2g}$ of $\pi_1(F)$.  Given a knot diagram, any Seifert surface obtained by the Seifert algorithm is considered regular; see \cite{Lin}. A knot $K$ is {\it fibered} if there is a fibration $\Sigma_{g,1} \stackrel{\iota}{\ra} S^3 \setminus K \ra S^1 $.

\begin{figure}[htpb]
\begin{center}
\begin{picture}(-60,60)
\put(-141,12){\pc{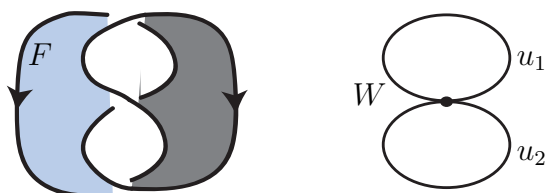}{0.533104}}
\put(-128,31){\large $F$}
\put(-5,16){\large $W $}
\put(57,31){\large $u_1 $}
\put(57,-5){\large $u_2 $}
\end{picture}
\end{center}
\caption{\label{ftft66} A regular Seifert surface and a regular spine. }
\end{figure}
According to \cite{Tro}, we should observe the complement $S^3 \setminus N(F)$ in which $N(F)$ is an open neighborhood of $F$ in $S^3$. Let $g \in \mathbb{N}$ be the genus of the surface $F$. We can easily see that $S^3 \setminus N(F)$ is homeomorphic to a handlebody of genus $ 2g $. Let us fix generators $x_1, \dots, x_{2g}$ of $\pi_1(S^3 \setminus N(F) ) $ and a $1$-handle $\mathcal{H}\subset S^3 \setminus L $ which represents a meridian $\mathfrak{m} $ of $K$. Consider the following subsets, as in Figure \ref{ftft66}:
$$ U= \mathcal{H}\cup N(F), \ \ \ \ \ \ \ \ \ \ V=S^3 \setminus F .$$
Notice that $ U, V$, and $U \cap V $ are homeomorphic to the handlebodies of genus $2g+1, 2g,$ and $4g$, respectively. Additionally, $U\cap V$ is $U \setminus F$ and consists of one layer above $F$ and one below joined by a bridge in $S^3 \setminus F$. Then, $\pi_1(U \cap V) $ has a canonical generating set $\{ u_1^{\sharp},\dots, u_{2g}^{\sharp}, u_1^{\flat }, \dots,u_{2g}^{\flat } \}$, where $u_{i}^{\sharp} $ is represented by a loop running parallel to $u_i$ in the upper layer, and $u_{i}^{\flat } $ is represented by a loop that runs parallel to $u_i$ in the lower layer. Lastly, $\{u_1^{\flat }, \dots,u_{2g}^{\flat } , \mathfrak{m} \}$ and $\{ u_1^{\sharp},\dots, u_{2g}^{\sharp} \}$ can be considered generators of $\pi_1( U)$ and $\pi_1( V) $, respectively.

We will give a presentation below \eqref{jjj}. Consider the inclusions $ i:U \cap V \hookrightarrow U$ and $i': U \cap V \hookrightarrow V$. The van Kampen theorem gives the following presentation of $\pi_1(S^3 \setminus K): $
\[ \langle \mathfrak{m} , x_1, x_2, \dots, x_{2g},u_1^{\flat }, \dots,u_{2g}^{\flat } \ | \ i_*(u_j^{\sharp})i_*'(u_j^{\sharp} )^{-1} , \ \ i_*(u_j^{\flat })i_*'(u_j^{\flat })^{-1}\ \ \ (1\leq j \leq 2g) \ \rangle \]
where $\mathfrak{m} $ is represented by the meridian of $K$. Since $i_* (u_j^{\flat } )$ can be written as a word of $x_1, \dots, x_{2g}$, the relator $u_j^{\flat }i_*'(u_j^{\flat } )^{-1}$ annihilates $ u_j^{\flat }= i_*' (u_j^{\flat })$. In addition, note that when $i_*(u_j^{\sharp})= \mathfrak{m} u_j^{\flat } \mathfrak{m}^{-1}$, the relator $i_*(u_j^{\sharp})i_*'(u_j^{\sharp})^{-1}$ turns out to be $ \mathfrak{m} \cdot i_*( u_j^{\flat})\cdot \mathfrak{m}^{-1}\cdot ( i_*( u_j^{\sharp}))^{-1} $. We denote $ i_*( u_j^{\flat})$ by $y_j$, and $ i_*( u_j^{\sharp} )$ by $z_j$ (where $y_j,z_j$'s are some words of $x_1, \dots, x_{2g}$).  Now, we conclude the presentation:
\begin{equation}\label{jjj} \pi_1(S^3 \setminus K) \cong \langle x_1, \dots , x_{2g}, \mathfrak{m} \ | \ \mathfrak{m} y_i \mathfrak{m}^{-1}= z_i \ \ \ \ i \in \{ 1, \dots , 2g\} \rangle .\end{equation}
This presentation is a folklore and appears in, e.g., \cite{C,Lin,Tro}.
\begin{figure}[htpb]
\begin{center}
\begin{picture}(-60,90)
\put(-231,32){\pc{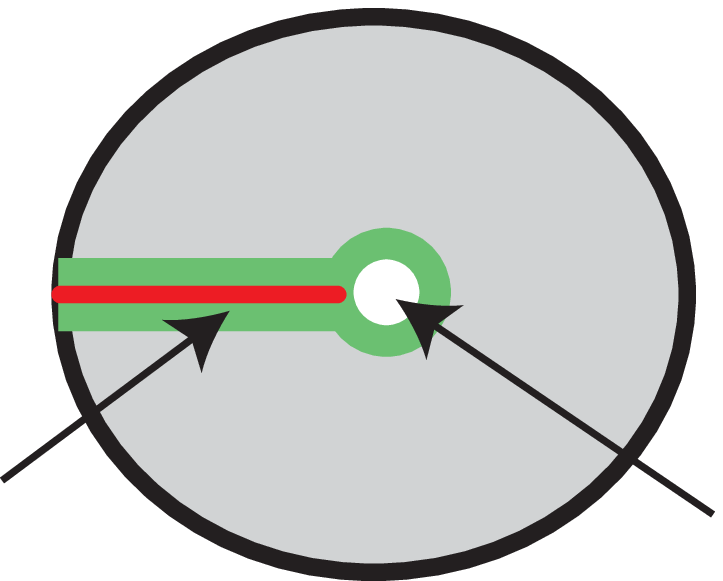}{0.430104}}
\put(-27,42){\pc{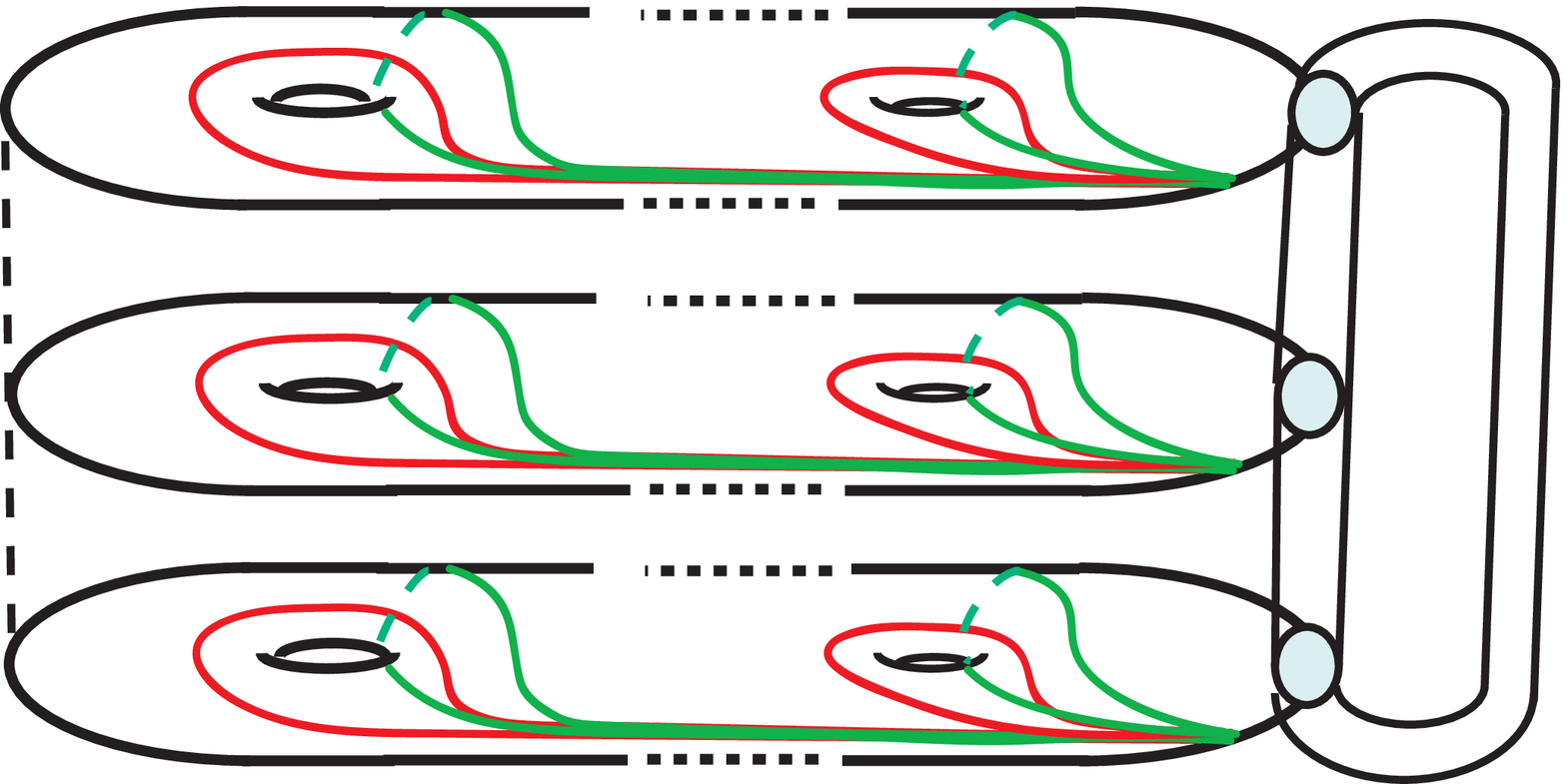}{0.330104}}
\put(-231,31){\large $U$}
\put(-138,32){\large $V$}
\put(-236,4){\large $F$}
\put(-141,-1){\large $K$}

\put(-11,4){\large $u_1^{\flat}$}
\put(44,4){\large $u_2^{\flat}$}
\put(114,4){\large $u_{2g}^{\flat}$}

\put(-11,39){\large $u_1$}
\put(44,39){\large $u_2$}
\put(114,39){\large $u_{2g}$}
\put(-35.2,39){\large $F$}
\put(174,39){\large $\mathcal{H}$}

\put(-11,73){\large $u_1^{\sharp}$}
\put(44,73){\large $u_2^{\sharp}$}
\put(114,73){\large $u_{2g}^{\sharp}$}

\put(-261,73){\large (a)}
\put(-64,73){\large (b)}
\end{picture}
\

\end{center}
\caption{\label{ftft} (a) $U$, $V$, and $F$. (b) representative generators $\{ u_1^{\sharp},\dots, u_{2g}^{\sharp}\} $ and $\{u_1^{\flat }, \dots,u_{2g}^{\flat }\} $ in $U$. }
\end{figure}

\section{Definition of the $K_1$-Alexander invariant}
\label{defS2}
This section introduces $K_1$-classes as a generalization of the Alexander polynomial and states Theorems \ref{thm11} and \ref{thm119}. The proofs of the theorems will appear in \S \ref{appeSS}.

We begin by defining algebraic terminologies. Let $\mathcal{A}$ be a ring with a unit, which may be non-commutative, and take a ring isomorphism $\kappa : \mathcal{A} \ra \mathcal{A} $. Then, we have the completed skew Laurent polynomial ring $\mathcal{A}_{\kappa}( \!( \tau) \!) $. More precisely, $\mathcal{A}_{\kappa}( \!( \tau) \!) $ is the set of formal power series $ \sum_{i= - N }^{\infty}a_i \tau^i $ where $a_i \in \mathcal{A}$ and $ \tau^n a = \kappa^n (a) \tau^n $. In other words, $\mathcal{A}_{\kappa}( \!( \tau) \!) $ is equal to $\mathcal{A}[ \![ \tau] \!][\tau^{-1}] $, which is called the Novikov ring in \cite{PR}. The reasons we consider Novikov rings instead of usual polynomial rings are as follows: first, the invertibility of matrices  over a Novikov ring is more common than those over polynomial rings (cf. studies of $K_1$; see \cite{PR}). Second, $1- \tau $ is invertible $\mathcal{A}_{\kappa}( \!( \tau) \!) $, which is an important property of knot theory. In addition, the semi-locality is suitable for $\mathcal{A}_{\kappa}( \!( \tau) \!) $, as seen in \S \ref{invdefS2}.

The following is assumed throughout this paper:

\vskip 0.61pc

\noindent
\underline{\bf Assumption $(\dagger)$}
Let $\mathcal{A} $ be a ring and $\kappa :\mathcal{A} \ra \mathcal{A}$ be 
a ring isomorphism. 
We fix a meridian $ \mathfrak{m} \in \pi_1(S^3 \setminus K) $ and suppose a ring homomorphism $\rho: \Z[ \pi_1(S^3 \setminus K)] \ra \mathcal{A}_{\kappa}( \!( \tau) \!) $ satisfying $\rho (\mathfrak{m} )= \tau $ and $ \rho (x_i) \in \mathcal{A} $, where $\mathfrak{m}, x_i$ are the generators in \eqref{jjj}.

\begin{exa}\label{exppp0}
We will see that any group homomorphism $ \pi_1(S^3 \setminus K) \ra G$ yields a situation satisfying $(\dagger)$ as follows.

Given a semi-direct product $H \rtimes \Z$ and a group homomorphism $h: \pi_1(S^3 \setminus K) \ra H \rtimes \Z$ such that $h ( \mathfrak{m}) =(1_G,0) $ and $h (x_i) \in H \times \{ 0\}$, we have the following situation with $(\dagger)$. Let $\mathcal{A} $ be the group ring $B[H]$ over a commutative ring $B$. If we replace $ h(\mathfrak{m})$ with $ \tau $, we can define $\mathcal{A}_{\kappa}( \!( \tau) \!) $. Then, $h$ canonically gives rise to $ \rho : \Z[ \pi_1(S^3 \setminus K)] \ra \mathcal{A}_{\kappa}( \!( \tau) \!) $, satisfying $(\dagger).$

Generally, we obtain such an $h$ from any group $G$ and any group homomorphism $f: \pi_1(S^3 \setminus K) \ra G $ as follows: let $ H$ be the subgroup generated by $f(x_i) $'s and let $\Z = \{ \tau^n \}_{n \in \Z} $ act on $H$ by $ g \cdot \tau^n:= f(\mathfrak{m} )^{n} g f(\mathfrak{m})^{-n} $. From this action, we can obtain the semi-direct product $H \rtimes \Z $ and can then define a homomorphism $ h: \pi_1(S^3 \setminus K) \ra H \rtimes \Z $ by $ h(x_i )= (f(x_i) ,0)$ and $ h( \mathfrak{m})=(1_H ,1)$.
\end{exa}

Let us set up $K_1$-groups. For a ring $R$ with unit, let $GL_n(R)$ be the general linear group over $R$ of size $n$. Since $ GL_n(R)$ canonically injects into $GL_{n+1}(R) $, we have the colimit $GL(R) = \lim GL_n(R) $. The {\it $K_1$-group}, $K_1(R)$, is defined to be the abelianization $ GL(R)/[GL(R),GL(R)]$. We often view elements of $GL_n(R)$ as those of $K_1(R)$. 
The inclusion $\mathcal{A} \hookrightarrow \mathcal{A}_{\kappa}( \!( \tau) \!) $ induces $ K_1(\mathcal{A})\ra K_1( \mathcal{A}_{\kappa}( \!( \tau) \!))$, which is known to be a splittable injection \cite{PR}. In this paper, we mainly consider the cokernel and denote it by $\mathcal{Q}_{\mathcal{A}, \kappa}$; that is,
\begin{equation}\label{oo}\mathcal{Q}_{\mathcal{A}, \kappa}:= K_1( \mathcal{A}_{\kappa}( \!( \tau) \!))/ K_1(\mathcal{A}). \end{equation}

Next, let us review the Fox derivative. Let $ \mathcal{F}$ be the free group of rank $2g$. Given a basis, $x_1,\dots, x_{2g}$, of $ \mathcal{F}$, we define a $\Z$-linear map $\frac{\partial \ }{\partial x_i} : \Z [ \mathcal{F}] \ra \Z[ \mathcal{F}]$ by the following identities:
$$\frac{\partial x_j }{\partial x_i}= \delta_{ij}, \ \ \ \ \frac{\partial (hk) }{\partial x_i}= \frac{\partial h }{\partial x_i}+ h \frac{\partial k}{\partial x_i} \ \ \ \ \ \ \ (h,k \in \mathcal{F}) . $$

Let us consider the case in which $ \mathcal{F}$ equals $ \langle x_1, \dots, x_{2g} | \ \rangle $, as in \S \ref{ganeaS} to explain Definition \ref{def1156} below. Recalling the presentation of $\pi_1(S^3 \setminus K) $ in \eqref{jjj}, let us consider the Jacobian matrix of the relation \eqref{jjj}. Namely, 
\begin{equation}\label{oo7} \Bigl\{ \tau \rho ( \frac{\partial y_j }{\partial x_i} )- \rho ( \frac{\partial z_j }{\partial x_i} ) \Bigr\}_{ 1 \leq i,j \leq 2g} \in \mathrm{Mat}(2g \times 2g, \mathcal{A}_{\kappa}( \!( \tau) \!)). \end{equation}
Let $A_{F,W}$ denote this matrix, which was introduced by \cite{Tro,Lin}. Since $ \mathcal{A}_{\kappa}( \!( \tau) \!) $ can be considered a local coefficient of $S^3 \setminus K$, we can define the 1-st homology $H_1( S^3 \setminus K ; \mathcal{A}_{\kappa}( \!( \tau) \!))$ of the local system. As stated in \cite[Proposition 6.1]{C}, it is known\footnote{Strictly speaking, the paper \cite{C} assumed that $\mathcal{A}$ is a skew field obtained from a fraction of an Ore domain. Nevertheless, we can easily verify that the same proof runs for $\mathcal{A}_{\kappa}( \!( \tau) \!) $ with Assumption ($\dagger$).} that this $A_{F,W}$ gives a finite presentation of the 1-st homology $H_1( S^3 \setminus K ; \mathcal{A}_{\kappa}( \!( \tau) \!))$. There is an exact sequence
\begin{equation}\label{jjj5655} 0 \ra
( \mathcal{A}_{\kappa}( \!( \tau) \!))^{2g} \xrightarrow{ \ A_{F,W} \ } ( \mathcal{A}_{\kappa}( \!( \tau) \!))^{2g}\lra H_1( S^3 \setminus K ; \mathcal{A}_{\kappa}( \!( \tau) \!)) \lra 0.
\end{equation}

\begin{defn}\label{def1156}
If $A_{F,W}$ is an invertible matrix, we define the {\it $K_1$-Alexander invariant (with respect to $\rho$)} to be the $K_1$-class of $\tau^{-g} A_{F,W} $ in $ \mathcal{Q}_{\mathcal{A}, \kappa} $. More precisely,
$$ \Delta_{\rho}^{K_1} := [ \tau^{-g} A_{F,W} ] \in \mathcal{Q}_{\mathcal{A}, \kappa} . $$
On the other hand, if $A_{F,W}$ is not invertible, we define $\Delta_{\rho}^{K_1}$ to be zero.
\end{defn}
In Section \ref{InvHandle}, we will show the invariance: To be precise,
\begin{thm}[{cf. \cite[Theorem 3.3]{Lin}}]\label{thm11}
The class $\Delta_{\rho }^{K_1}$ in $\mathcal{Q}_{\mathcal{A}, \kappa} $ does not depend on the choice of the generator $x_1,\dots, x_{2g}$, the spine $W$, or the regular Seifert surface $F$. In other words, $\Delta_{\rho }^{K_1}$ depends only on the ring homomorphism $\rho:\Z[ \pi_1(S^3 \setminus K) ] \ra \mathcal{A}_{\kappa}( \!( \tau) \!) $ with Assumption ($\dagger $).
\end{thm}

\begin{rem}\label{rem11} 
The assumption of invertibility of $ A_{F,W}$ is not strong in many cases. For example, consider the Jacobson radical $J(R ) $ of a ring $R$, i.e., $ J(R)=\{\,x\in R \mid 1+R x R\subset R^{\times }\,\}$. A matrix $B \in \mathrm{Mat}(n \times n, R) $ is invertible if and only if the quotient of $B$ in $\mathrm{Mat}(n \times n, R/J (R)) $ is invertible as well  (see \cite[Propitiation III.2.2 and Corollary III.2.7]{Bass2}). In conclusion, studying the quotient of $A_{F,W}$ in $\mathrm{Mat}(n \times n, \mathcal{A}_{\kappa}( \!( \tau) \!) / J(\mathcal{A}_{\kappa}( \!( \tau) \!) )) $ is sufficient for checking the invertibility of $A_{F,W}$. In particular, when the quotient $\mathcal{A}_{\kappa}( \!( \tau) \!) / J(\mathcal{A}_{\kappa}( \!( \tau) \!) )$ is semi-simple as in \S \ref{invdefS2}, we can easily verify that $A_{F,W}$ is invertible in many cases.
\end{rem}
Furthermore, we will show 
a criterion for fiberedness.
\begin{thm}\label{thm119}
A knot $K$ is fibered if and only if $ A_{F,W}$ is invertible for any $\rho$ satisfying $(\dagger)$.
\end{thm}
\begin{rem}\label{thm11934}
Some criterions have been established for fiberedness in terms of (twisted) Alexander polynomials or Novikov rings; see \cite{Lic,FV,FV2}. As seen in \S \ref{wwS}, the proof of this theorem is based on a theorem in \cite{Fri}.

In the paper \cite{GS}, the authors defined $K_1$-valued topological invariants with respect to fibered spaces $Y \ra S^1$. 
In the last page, they pose a question to make knot-invariants, which lay in a $K_1$-group.
Thus, Definition \ref{def1156} gives a $K_1$-Alexander invariant, which counts as a generalization to the non-fibered case.
\end{rem}

We conclude this section by explaining that the twisted Alexander polynomial of \cite{Lin,Wada} can be formulated from our $K_1$-value invariant:
\begin{exa}\label{exa1124}
The papers \cite{Lin,Wada} defined a polynomial from a representation $\rho^{\rm pre}: \pi_1(S^3 \setminus K) \ra GL_n(R)$, in which $R$ is a commutative Noetherian unique factorization domain. For the formulation, let $\mathcal{A} $ be the matrix ring $\mathrm{Mat}(n \times n,R)$ and let $\kappa$ be the identity $\mathrm{id}_{\mathcal{A} }$. Formally, let $\tau$ be $\rho_{\rm pre} (\mathfrak{m})$ as a commutative indeterminate. Since there is a ring homomorphism $ \Z[ GL_n(R)] \ra \mathrm{Mat}(n \times n,R) $ which sends $\sum_{g} a_g g$ to $\sum_{g} a_g g$, the $ \rho^{\rm pre} $ gives rise to a ring homomorphism $\rho:\Z[\pi_1(S^3\setminus K) ] \ra \mathcal{A}_{\kappa}( \!( \tau) \!)$. Notice that the determinant $\mathrm{Mat}(n \times n,R ( \!( \tau) \!)) \ra R ( \!( \tau) \!)$ induces homomorphisms
$$K_1(\mathcal{A}_{\kappa}( \!( \tau) \!) ) \lra R ( \!( \tau) \!)^{\times} \ \ \ \ \ \mathrm{and} \ \ \ \ \ \mathcal{Q}_{\mathcal{A}, \kappa} \lra R ( \!( \tau) \!)^{\times} / R^{\times}.$$
After thoroughly checking \cite[Definition 3.2]{Lin}, we verify, by construction, that the determinant $\det (\Delta_{\rho}^{K_1} ) \in R ( \!( \tau) \!)^{\times} / R^{\times}$ coincides exactly with the twisted Alexander polynomial of \cite{Lin}. In the case of a knot, the twisted Alexander polynomial of \cite{Lin} is known to be equal to that of \cite{Wada}. In other words, our $K_1$-value $\Delta_{\rho}^{K_1}$ is a lift of the known twisted Alexander polynomials.
\end{exa}

\begin{exa}\label{thm1441}
Let $\mathcal{A}$ be $ \Q$, let $\kappa$ be $\mathrm{id}_{\Q}$, and let $\iota: \Z[t^{\pm 1}] \hookrightarrow \Q_{\kappa} ( \!( \tau) \!) $ be the ring monomorphism such that $\iota(t)= \tau$. Define $\rho$ to be the composite $\iota \circ \mathrm{Ab} : \Z[\pi_1(S^3 \setminus K)] \ra \Q_{\kappa} ( \!( \tau) \!) $. Then, the determinant $\Delta_{\rho}$ is exactly equal to the definition of the classical Alexander polynomial; see \cite[Chapter 6]{Lic} or \cite{Tro}. Thus, our $\Delta_{\rho} $ is also a generalization of the classical Alexander polynomial.
\end{exa}

\section{Two studies of the container $\mathcal{Q}_{\mathcal{A}, \kappa}$}
\label{QSS}
In general, it is difficult to compute a $K_1$-group and show non-triviality of an element in the group. However, using the works of \cite{PR,HKL} as our basis, we suggest two procedures for obtaining quantitative information from the group $\mathcal{Q}_{\mathcal{A}, \kappa}$.

\subsection{Logarithm of the container $\mathcal{Q}_{\mathcal{A}, \kappa}$}
\label{subQSS1}
To quantitatively study $\mathcal{Q}_{\mathcal{A}, \kappa}$, let us review the decomposition theorems of $K_1$ and the logarithm from $\mathcal{Q}_{\mathcal{A}, \kappa}$, due to Pajitnov \cite{P} and Ranicki \cite{PR}. 
In addition to previously established information, this subsection contains two new propositions \ref{ww2} and \ref{ww}. 

We will introduce two groups. First, let $W_1(\mathcal{A},\kappa) \subset K_1( \mathcal{A}[ \![ \tau] \!]) $ be the subgroup represented by Witt vectors; that is, the units in $ \mathcal{A}_{\kappa}( \!( \tau) \!) $ of the type $ w= 1 +\sum_{j=1}^{\infty} a_j \tau ^j\in \mathcal{A}_{\kappa}( \!( \tau) \!)^{\times }$. Next, a $\kappa$-endomorphism $\nu: P \ra P $ is said to be {\it nilpotent} if the functional power $\nu^k$ is the zero map for some $k \in \mathbb{N}.$ Then, we can define the exact category with objects
$$ (Q: \textrm{f.g. projective } \mathcal{A}\textrm{-module}, \ \ \ \ \ \nu : Q \ra Q \ : \textrm{nilpotent } \kappa \textrm{-endomorphism}). $$
Then, let us define $\mathrm{Nil}_0(\mathcal{A}, \kappa )$ to be the Grothendieck group of the exact category, and $K_0(\mathcal{A}  $ to be that of the category of f.g. projective $\mathcal{A}$-module. Furthermore, the {\it reduced nilpotent class group $\widetilde{\mathrm{Nil}}_0(\mathcal{A}, \kappa ) $} is defined to be
$$ \widetilde{\mathrm{Nil}}_0(\mathcal{A}, \kappa ) := \mathrm{Coker}(K_0(\mathcal{A}) \lra
\mathrm{Nil}_0(\mathcal{A}, \kappa ) ) $$
with $K_0(\mathcal{A}) \ra \mathrm{Nil}_0(\mathcal{A}, \kappa ) ; [Q] \mapsto [Q,0]$. This map is known to be splittable.

Let us review the main theorem in \cite{PR}. Consider the following map:
\begin{equation}\label{jjj555} \widehat{C}_1 \oplus \widehat{C}_2 \oplus\widehat{C}_3: K_1(\mathcal{A})\oplus W_1(\mathcal{A},\kappa) \oplus \widetilde{\mathrm{Nil}}_0(\mathcal{A}, \kappa^{-1} ) \lra K_1( \mathcal{A}_{\kappa}( \!( \tau) \!)) \end{equation}
defined by setting
\[ \widehat{C}_1: K_1(\mathcal{A}) \ra K_1( \mathcal{A}_{\kappa}( \!( \tau) \!)) ; \ \ M \longmapsto [\tau M : \mathcal{A}_{\kappa}( \!( \tau) \!) \ra \mathcal{A}_{\kappa}( \!( \tau) \!)] ,\]
\[ \widehat{C}_2: W_1(\mathcal{A},\kappa) \lra K_1( \mathcal{A}_{\kappa}( \!( \tau) \!)) ; \ \ w \longmapsto w, \]
\[\widehat{C}_3: \widetilde{\mathrm{Nil}}_0(\mathcal{A}, \kappa^{-1} )\lra K_1( \mathcal{A}_{\kappa}( \!( \tau) \!)) ; \ \ (Q,\nu) \longmapsto [1-\tau^{-1} \nu : Q( \!( \tau) \!) \ra Q ( \!( \tau) \!)]. \]
The main theorem in \cite{PR} shows that \eqref{jjj555} is an isomorphism by constructing an explicit inverse of $ \widehat{C}_1 \oplus \widehat{C}_2 \oplus\widehat{C}_3 $. From the definition of $\mathcal{Q}_{\mathcal{A}, \kappa}$, we obtain the isomorphism
\begin{equation}\label{jjj55} \widehat{C}_2 \oplus\widehat{C}_3: W_1(\mathcal{A},\kappa) \oplus  \widetilde{\mathrm{Nil}}_0(\mathcal{A}, \kappa^{-1} ) \cong \mathcal{Q}_{\mathcal{A}, \kappa}. \end{equation}
Here, it is known that $ \widehat{C}_2 ^{-1}(\tau) $ is zero in $W_1(\mathcal{A},\kappa). $

Then, we will show a criterion for fiberedness (see \S \ref{wwS4} for the proof).
\begin{thm}\label{ww2} 
Suppose that $K$ is a fibered knot. The matrix $A_{F,W}$ is diagonal, and $\tau^g \Delta_{ \rho}^{K_1}$ is contained in $W_1( \mathcal{A}, \kappa)$, for any $\rho$ with $(\dagger)$.
\end{thm}

In addition, we review a logarithm from $W_1( \mathcal{A},\kappa)$ defined in \cite{P}. By definition, $W_1( \mathcal{A},\kappa)$ is equal to the image $\mathrm{Im}(W(\mathcal{A},\kappa ) \hookrightarrow \mathcal{A}_{\kappa}( \!( \tau) \!)^{\times } ) $, where the inclusion is given by
\begin{equation}\label{jjj775588} W(\mathcal{A}, \kappa) = \{ 1+a_1 \tau +a_2 \tau^2 + \cdots \ | a_i \in \mathcal{A} \ \} \subset \mathcal{A}_{\kappa}( \!( \tau) \!)^{\times} \end{equation}
as  a subgroup. There is a natural surjective homomorphism,
$$ J: W( \mathcal{A}, \kappa)_{\rm ab} \lra W_1( \mathcal{A}, \kappa) .$$
For $n\in\mathbb{Z}_{\geq 0}$, let $P_n $ be the left $ \mathcal{A} $-module $ \mathcal{A} \tau^n \subset \mathcal{A}_{\kappa}( \!( \tau) \!)$. Let $P_n' $ be the abelian subgroup of $P_n$ generated by all the commutators $xy - yx$ where $ x \in P_k,y\in P_{n-k}$. Set
$$ \bar{P}_n=P_n/ P_n', \ \ \ \ P'= \prod_{n \geq 0} P_n', \ \ \ \ \ \bar{P}= \prod_{n \geq 0} \bar{P}_n=P/P' $$
as abelian groups. In general, the quotient $ \bar{P}_n=P_n/ P_n'$ is prone to being small or zero by non-commutativity. However,
\begin{exa}\label{op2} 
If $ \kappa^N= \mathrm{id}_{\mathcal{A}}$ for some $N \in \mathbb{N}$, then $ \bar{P}_N=P_N/ P_N'$ is isomorphic to the quotient abelian group,
\begin{equation}\label{jjj7755} \mathcal{A}/\langle x y - \kappa^m (yx) \rangle_{x,y \in \mathcal{A}, \ m \in \mathbb{Z}}, \end{equation}
by definition. If $\mathcal{A} $ is commutative, then $\bar{P}_N= \mathcal{A}/ \{a - \kappa (a) \}_{a \in \mathcal{A}} $. For example, if $\mathcal{A} $ is a commutative $\Q$-algebra, then $ \mathrm{dim}_{\Q} (\bar{P}_N) \geq \mathrm{dim}_{\Q}(\mathcal{A} )/N$, which shows the non-triviality of $ \bar{P}_N $.
\end{exa}
\noindent
Moreover, provided $\Q \subset \mathcal{A}$, let us define $\mathrm{log}: W(\mathcal{A}, \kappa) \ra \bar{P} $ as
$$ \mathrm{log}(1 + \mu\tau) = \mu \tau - \frac{(\mu\tau)^2}{2}+ \frac{(\mu\tau)^3}{3}- \cdots (-1)^{n-1} \frac{(\mu\tau)^n}{n}+ \cdots, \ \ \ \ \mathrm{where} \ \ \mu \in \mathcal{A}_{\kappa}[ \![ \tau] \!]. $$
It is shown in \cite[Lemma 1.1]{P} that $\Ker (J) \subset \Ker (\mathrm{log})$ and that the induced map $ \mathrm{Log} : W_1(\mathcal{A},\kappa) \ra \bar{P} $ is a homomorphism.

Using the logarithm, we will show a knot invariance (the proof will appear in \S \ref{wwS}):
\begin{prop}\label{ww} 
Suppose $\tau^g \Delta_{ \rho }^{K_1 } $ lies in $W_1( \mathcal{A}, \kappa)$. The $\tau^k$-coefficient of $\mathrm{Log } ( \tau^g \Delta_{ \rho}^{K_1 } ) $ in $\bar{P}_k$ does not depend on the choice of the meridian $\mathfrak{m}$, as the $\tau^k$-coefficient is an invariant of $\rho$ .
\end{prop}

\subsection{Morita invariance on $K_1$ and the container $\mathcal{Q}_{\mathcal{A}, \kappa}$}
\label{subQSS2}
In this subsection, we use the Morita invariance of $K_1$ to construct another map from $\mathcal{Q}_{\mathcal{A}, \kappa}$. Generally, because $\mathcal{Q}_{\mathcal{A}, \kappa}$ with $ \kappa \neq \mathrm{id}_{\mathcal{A}}$ is complex, as seen in \eqref{jjj7755}, we shall take up a method to replace $\kappa$ with $\mathrm{id}_{\mathcal{A}}$. In order to do so, we suppose the existence of $N \in \mathbb{N}$ such that $\kappa^N= \mathrm{id}_{\mathcal{A}}$ throughout this subsection.

 Let us first define a ring homomorphism below \eqref{ooo2}. For $\ell \leq N$ and $a \in \mathcal{A}$, we set up a diagonal $(\ell \times \ell)$-matrix of the form
\[ \mathcal{D}_\ell (a):= \left(\begin{array}{rrrr} \kappa^\ell (a) & 0 & \cdots & 0\\
0 & \kappa^{\ell -1}(a) & \cdots & 0 \\
\vdots & \vdots & \ddots & \vdots \\
0 & 0 & \cdots & \kappa^{1}(a)
\end{array}\right) , \]
and denote the zero $(n \times m)$-matrix by $\mathbb{O}_{n,m} $. A square matrix is defined as
$$ M_{\ell}(a):= \left(\begin{array}{rr} \mathbb{O}_{N - \ell ,\ell} & \ \mathcal{D}_{N-\ell} ( \kappa^{\ell +1}(a)) \\
\mathcal{D}_\ell (a) & \!\!\!\!\!\mathbb{O}_{\ell, N - \ell}
\end{array}\right) \in \mathrm{Mat}(N \times N , \mathcal{A}) . $$
Let $t$ be a commutative indeterminate. Consider the Laurent polynomial ring $\mathcal{A}_{\rm id }( \!( t ) \!) $. Next, we introduce a map $\Upsilon$ defined by setting
\begin{equation}\label{ooo2}\Upsilon: \mathcal{A}_{\kappa}( \!( \tau) \!) \lra \mathrm{Mat}(N \times N ,\mathcal{A}_{\rm id }( \!( t ) \!)); \ \ \ \ \ \sum_{i\geq 1} a_i \tau^i \longmapsto \sum_{i \geq 1} M_{i}(a_i) t^i . \end{equation}
\begin{exa}\label{oo3}
Assume $N=3$. For $a,b,c \in \mathcal{A}$, the target $\Upsilon( a+ b\tau + c\tau ^2)$ is formulated as
\[
\left(
\begin{array}{rrr} \kappa^3(a) & 0 & 0 \\
0 & \kappa^2(a) & 0\\
0& 0&\kappa(a) \\
\end{array}
\right) t^0 + \left(
\begin{array}{rrr} 0 & 0 & \kappa^3(b) \\
\kappa^2(b) & 0 & 0\\
0& \kappa^1(b) &0 \\
\end{array}
\right)t + \left(
\begin{array}{rrr} 0 & \kappa^3(c) & 0\\
0 & 0 & \kappa^2(c) \\
\kappa^1(c) &0 &0\\
\end{array}
\right)t^2.
\]
\end{exa}
Then, we can easily verify the following proposition:
\begin{prop}\label{ooo244} 
The map $ \Upsilon$ is a ring homomorphism.
\end{prop}
\noindent
Therefore, the homomorphism $ \Upsilon$ induces
$$ \Upsilon_* : K_1 (\mathcal{A}_{\kappa}( \!( \tau) \!) ) \lra K_1 (\mathrm{Mat}(N \times N , \mathcal{A}_{\rm id }( \!( t ) \!)) ), \ \ \ \ $$
Furthermore, the Morita invariance of $K_1$ (see, e.g., \cite[Lemma 1.10]{Mil} or \cite[Example 1.1.4]{Wei}) implies the isomorphism
\begin{equation}\label{ooo621}\mathcal{M}: K_1 (\mathrm{Mat}(N \times N , \mathcal{A}_{\rm id }( \!( t ) \!) ) )\cong
K_1 ( \mathcal{A}_{\rm id }( \!( t ) \!)) . \end{equation}
According to the decomposition \eqref{jjj555}, we can see that the composite $ \mathcal{M} \circ \Upsilon_* $ descends to
$$\mathcal{M} \circ \Upsilon_* : \mathcal{Q}_{\mathcal{A} , \kappa} \lra \mathcal{Q}_{\mathcal{A}, {\rm id} }, \ \ \ \ \ \ W_1(\mathcal{A} , \kappa) \lra W_1(\mathcal{A} , {\rm id} ). $$
\begin{exa}\label{oo66}
Suppose that $\mathcal{A}$ is a commutative ring. The isomorphism \eqref{ooo621} is represented by the determinant map. Moreover, if the determinant map induces the isomorphism $ K_1( \mathcal{A}_{\rm id }( \!( t ) \!) ) \cong \mathcal{A}_{\rm id }( \!( t ) \!) ^{\times }$ as in semi-local rings (see \S \ref{invdefS2}), the previous composite $\mathcal{M} \circ \Upsilon_* $ induces a homomorphism,
\begin{equation}\label{ooo6} \mathrm{det} \circ \Upsilon_* : \mathcal{Q}_{\mathcal{A} , \kappa} \lra \mathcal{A}_{\rm id }( \!( t ) \!) ^{\times } / \mathcal{A}^{\times}. \end{equation}
In summary, $\mathrm{det} \circ \Upsilon_* ( \Delta_{\rho}^{K_1})$ lies in a commutative object in such a situation. This will be studied with relation to cyclic covering spaces; see \cite{Nos}.
\end{exa}
Finally, we will show independence from meridians $\mathfrak{m}$ (see \S \ref{wwS} for the proof).
\begin{prop}\label{ww3} 
Suppose $ N $ such that $\kappa^N = \mathrm{id}$. Then, the pushforward of the $K_1$-invariant $ \mathcal{M} \circ \Upsilon_*( \Delta_\rho^{K_1})$ is independent of the choice of the meridian $\mathfrak{m}.$
\end{prop}

\section{$K_1$-Alexander polynomials from semi-local situations}
\label{invdefS2}
We will consider situations in which the invariant $\Delta_{\rho}^{K_1} $ can be studied from the viewpoint of the determinant.

We start by reviewing semi-local rings. A ring $R$ is {\it semi-local} if $ R/J(R)$ is a semi-simple ring (in which $J(R)$ is the Jacobson radical). A ring $R$ is {\it semi-perfect}, if $R$ is semi-local and every idempotent of $R/J(R)$ can be lifted to $R$. For example, it is known \cite{Zie} that, if $ \mathcal{A}$ is semi-perfect, so is $ \mathcal{A}_{\kappa}( \!( \tau) \!)$.
But the converse is not true; it is worth of noting the fact \cite{Solin} that if the skew Laurent polynomial ring $\mathcal{A}_{\kappa}( \!( \tau) \!) $ is semi-local, then $ \mathcal{A}$ is semi-perfect and $J(\mathcal{A})$ is equal to the nilpotent radical. 

Next, we review the Whitehead determinant. Let $R$ ba a semi-local ring with unit $1$. If $r,s \in R$ satisfy $1+rs \in R^{\times}$, so does $1+sr$ because $(1+sr)(1-s(1+rs)^{-1}r)=1$. Let $\mathcal{V}(R)$ denote the subgroup of $R^{\times} $ generated by $ (1+rs )(1+sr )^{-1}$. We can easily see that $\mathcal{V}(R)$ contains the commutator subgroup $[R^{\times }, R^{\times }]$. For a generalization of the Dieudonn\'{e} determinant, there uniquely exists a determinant map \cite[\S 2]{Bass}, 
$$ \mathrm{det}: \mathrm{Mat}(n \times n, R) \lra \bigl( R^{\times }/\mathcal{V}(R) \bigr) \cup \{0\},$$
with the following properties:
\begin{enumerate}[(i)]
\item $ \det (A) \det (B)= \det (A \cdot B)$ holds for any $n \times n$-matrices $A,B$ over $ R $, where the dot $\cdot $ represents the multiplication of matrices.
\item The determinant is invariant under elementary row operations.
\item The determinant of the identity is 1.
\item If a matrix $B$ is invertible, then $\mathrm{det}(B) $ lies in $ R^{\times }/\mathcal{V}(R) $; otherwise $\mathrm{det}(B)=0$.
\end{enumerate}
See also, e.g., \cite{Vas3,Vas} or \cite[\S 3.1]{Wei} for the definition. It is shown that if $R$ is semi-local, the inclusion $\lambda: R^{\times} =GL_1(R) \hookrightarrow GL (R)$ gives rise to an isomorphism $\lambda_* : R^{\times}/ \mathcal{V}(R) \cong K_1(R) $ and the Whitehead determinant $ GL(R) \ra R^{\times} / \mathcal{V}(R) $ yields $ K_1(R) \cong R^{\times} /\mathcal{V}(R)$ as well (see \cite{Vas3}). Therefore, when $R= \mathcal{A}_{\kappa}( \!( \tau) \!) $ is semi-local, the quotient group $\mathcal{Q}_{\mathcal{A}, \kappa} $ in \eqref{oo} is, by definition, isomorphic to
\begin{equation}\label{ooo}\mathcal{Q}_{\mathcal{A}, \kappa}^{\det }:= \frac{ \mathcal{A}_{\kappa}( \!( \tau) \!)^{\times }/\mathcal{V}( \mathcal{A}_{\kappa}( \!( \tau) \!)) }{\mathcal{A}^{\times }/\mathcal{V}( \mathcal{A}) }. \end{equation}
\begin{defn}\label{def11} 
Let $\rho : \Z[ \pi_1(S^3 \setminus K)] \ra \mathcal{A}_{\kappa}( \!( \tau) \!)$ satisfy the assumption $(\dagger)$. Suppose that $ \mathcal{A}_{\kappa}( \!( \tau) \!)$ is semi-local. We define the {\it semi-local Alexander polynomial (with respect to $\rho$)} to be the class of the determinant of $\tau^{-g} A_{F,W} $ in $ \mathcal{Q}_{\mathcal{A}, \kappa} $. More precisely,
$$ \Delta_{\rho} := [\det \bigl( \tau^{-g} A_{F,W} \bigr)] \in \mathcal{Q}_{\mathcal{A}, \kappa}^{\det } \cup \{0\}. $$
\end{defn}
Next, we will give examples showing that $ \mathcal{A}$ is semi-local and $\rho$ satisfies the assumption $(\dagger)$

\begin{exa}\label{exppp44}
For an Artinian ring $R$ (e.g., $R$ is a commutative field) in a generalized situation of \cite{Lin}, we fix a representation $ \rho^{\rm pre}: \pi_1(S^3\setminus K) \ra GL_n(R)$. Let $\mathcal{A} $ be the matrix ring $\mathrm{Mat}(n \times n,R)$, which is an Artinian ring, and let $\kappa$ be the identity $\mathrm{id}_{\mathcal{A} }$. Formally, let $\tau$ be $\rho^{\rm pre} (\mathfrak{m})$ as a commutative indeterminate. As in Example \ref{exa1124}, this $ \rho^{\rm pre} $ gives rise to a ring homomorphism $\rho:\Z[\pi_1(S^3\setminus K) ] \ra \mathcal{A}_{\kappa}( \!( \tau) \!)$, which satisfies Assumption ($\dagger$).
\end{exa}

\begin{exa}\label{exppp35}
The following is an example in a metafinite sense. We shall mention the fact:
\begin{prop}[\cite{Wood}]\label{expppop35} 
Let $B$ be a semi-local commutative ring and $G$ a group. If one of the following holds, the group ring $B[G]$ is semi-perfect.
\begin{enumerate}[(a)]
\item $B$ is a field and $G$ is of finite order.
\item $B$ is a perfect local ring with char($B$) = $p > 0$, $G$ is a locally finite group, and $G$ has a $p$-subgroup of finite index. 
\end{enumerate}
\end{prop}
\noindent
In addition, let $ \pi_K'$ be the commutator subgroup of $ \pi_1(S^3 \setminus K)$ and $H$ a finite group acted on by $\Z$. Let $B$ be one of the following: $\Q$, $\Z/m\Z$, or the $p$-adic integer $\Z_p$. As seen in Example \ref{exppp0}, if we have a $\Z$-equivariant epimorphism $ \rho^{\rm pre} : \pi_K' \ra H$, we have $\rho: \pi_1(S^3 \setminus K) \ra G $, where we let $ G$ be  $ H \rtimes \Z$. Then, $\mathcal{A} =B [H]$ is semi-perfect by Proposition \ref{expppop35} and $\rho$ satisfies the Assumption $(\dagger)$.
\end{exa}
\begin{defn}\label{def133} 
Let $ G=H \rtimes \Z$ and $\mathcal{A} =B [H]$ be as in Example \ref{exppp35}, where $H$ is a finite group. Let $ N \in \Z_{>0}$ be the minimal number satisfying $ \kappa^N= \mathrm{id}$. We define the {\it metafinite Alexander polynomial (associated to $\rho^{\rm pre}$)} to be the polynomial
$$ \Delta_{ \rho^{\rm pre}_H}^{\Upsilon} := \mathrm{det} \circ \Upsilon_* \bigl( \Delta^{K_1}_{ \rho} \bigr) \in
\mathcal{Q}_{\mathcal{A}, \kappa}^{\det } \cup \{0\} . $$
\end{defn}

Finally, we will see that this definition includes the (metabelian) twisted Alexander polynomials in \cite{HKL}. Let us review the setting in \cite{HKL}. Fix $N \in \mathbb{N}$ and take the cyclic $N$-fold covering space $E_K^N$ of $S^3 \setminus K$. Let $ H$ be the torsion subgroup of the homology $ H_1(E_K^N;\Z)$. Then, we canonically have a group homomorphism $\rho^{\rm meta}_{ N}: \pi_1(S^3 \setminus K ) \ra H \rtimes \Z$.

\begin{defn}\label{def13312} 
Take $N \in \mathbb{N}$. In the above situation, as in $\mathcal{A}=\Q[H]$ with $H=\mathrm{Tor}H_1(E_K^N;\Z) $, we define the {\it $N$-fold metabelian Alexander polynomial} to be the polynomial,
$$ \Delta_{ \rho^{\rm meta}_{ N} }^{\Upsilon} := \mathrm{det} \circ \Upsilon_* \bigl( \Delta^{K_1}_{ \rho^{\rm meta}_{ N}} \bigr) \in
\mathcal{A}_{\rm id }( \!( t ) \!) ^{\times } / \mathcal{A}^{\times} \cup \{0\} . $$
\end{defn}
\begin{rem} 
Given a homomorphism $\chi: \mathrm{Tor}H_1(E_N;\Z) \ra \Z/q $ for some $q \in \mathbb{N}$, the pushforward $\chi_* ( \Delta_{ \rho^{\rm meta}_{ N} }^{\Upsilon} )$ coincides exactly with the twisted Alexander polynomial, which is defined in \cite[\S 7]{HKL}. Thus, Definitions \ref{def133} and \ref{def13312} are a slight generalization of the polynomials.
\end{rem}
By Proposition \ref{ww3}, the polynomial $\Delta_{ \rho^{\rm meta}_{ N} }^{\Upsilon} $ turns out to be a knot invariant.

In \S \ref{ExaSSs2}, we will compute some $\Delta_{ \rho^{\rm meta}_{N} }^{\Upsilon} $ and show the non-triviality of $ \Delta_{ \rho^{\rm meta}_{N} }^{\Upsilon} $ for some cases.

\section{Examples of $K_1$-Alexander invariants}
\label{ExaSS}
This section contains the computations for the $K_1$-Alexander invariants $\Delta_{\rho}^{K_1}$ of the trefoil knot, figure eight knot, and $5_2$-knot. To compute $\Delta_{\rho}^{K_1}$, the words $y_i$ and $z_i$ in \eqref{jjj} must be concretely defined first. If $K$ is a fibered knot with crossing number $<11$, the monodoromy $\phi:\pi_1 (F) \ra \pi_1(F)$ is presented in KnotInfo \cite{CL}. In this case, since $y_i$ is equal to $ x_i$ and $z_i =\phi(x_i)$, we can obtain the presentation as in \eqref{jjj} for such fibered knots. On the other hand, for a non-fibered knot, \cite{GS} suggests an algorithm for obtaining such a presentation from a Seifert surface.

\begin{figure}[htpb]
\begin{center}
\begin{picture}(10,60)
\put(-161,22){\pc{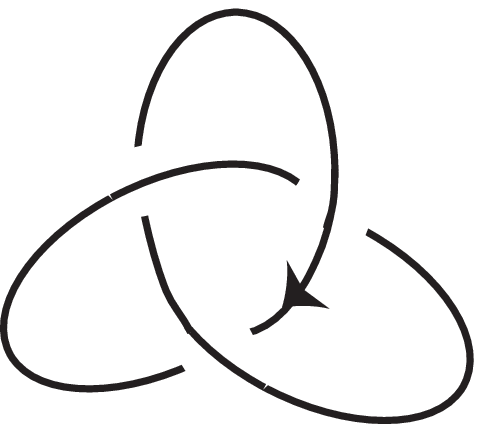}{0.53104}}
\put(-36,22){\pc{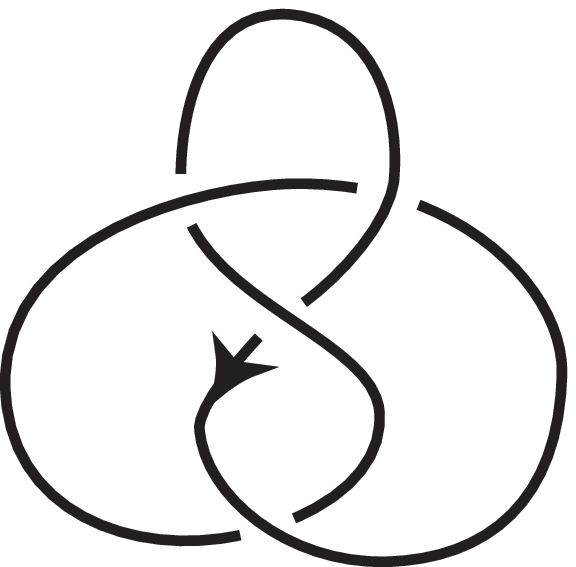}{0.402104}}
\put(86,22){\pc{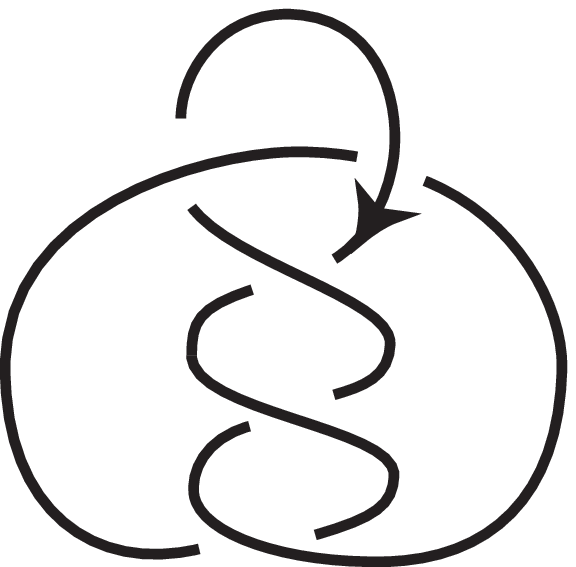}{0.402104}}
\end{picture}
\end{center}
\caption{\label{ftf} The trefoil knot, figure eight knot, and $5_2$-knot.}
\end{figure}

\subsection{Trefoil knot}
\label{ExaSSs}
First, let us focus on the trefoil knot $3_1$. According to KnotInfo \cite{CL}, we have the presentation,
$$ \pi_1(S^3 \setminus 3_1) \cong \langle x_1, x_2, \mathfrak{m} \ | \ \mathfrak{m} x_1 x_2^{-1} \mathfrak{m}^{-1}= x_1, \ \mathfrak{m} x_2 \mathfrak{m}^{-1}= x_2 x_1^{-1} \ \ \rangle. $$
Then, from the definition of $A_{F,W}$ in \eqref{oo7}, we notice
$$A_{F,W}= \tau \left(
\begin{array}{rr}
1 & -\rho( x_1 x_2^{-1} ) \\
0& \!\!\!\! 1\\
\end{array}
\right) - \left(
\begin{array}{rr} \!\!\!\! 1 & 0 \\
-\rho ( x_2 x_1^{-1} )& 1\\
\end{array}
\right)= \left(
\begin{array}{rr} \tau -1 & -\tau \rho (x_1 x_2^{-1} ) \\
\rho (x_2 x_1^{-1} ) & \tau- 1\\
\end{array}
\right) .$$
We then observe the equality,
\[ \left(
\begin{array}{rr} 1& 0\\
(\tau-1) \rho (x_1 x_2^{-1} ) & 1 \\
\end{array}
\right) \left(
\begin{array}{rr} 0& 1\\
-1 & 0 \\
\end{array}
\right)
A_{F,W} \left(
\begin{array}{rr} 1& \rho (x_1 x_2^{-1} ) (\tau-1) \\
0 & 1 \\
\end{array}
\right)\]
\[= \left(
\begin{array}{rr} -\tau \rho( x_2 x_1^{-1})& 0 \\
0 & (\tau^{-1}
-1 +\tau )\rho( x_1 x_2^{-1}) \\
\end{array}
\right) .\]
According to the Whitehead lemma (see \cite[Lemma III.1.3.3]{Wei}), any elementary matrix is zero in the $K_1$-group. Thus, elementary matrices are zeros in $K_1$, so the $K_1$-invariant turns out to be 
\[ \Delta_{\rho}^{K_1}= [\tau^{-1 } A_{F,W} ]= \tau^{-1 } \tau \rho (x_2 x_1^{-1} ) (\tau^2 -\tau +1) \rho (x_1 x_2^{-1}) = \tau -1 +\tau^{-1 } \in \mathcal{Q}_{\mathcal{A}, \kappa} . \] 
Formally, $\Delta_{\rho}^{K_1}$ does not depend on $\rho$. Since $\tau \Delta_{\rho}^{K_1} $ lies in $W(\mathcal{A},\kappa)$, we can obtain the logarithm $ \mathrm{Log}(\tau \Delta_{\rho})$, as in \S \ref{subQSS1}. Note that $ \kappa^6= \mathrm{id}_\mathcal{A}$. Therefore, if $\mathcal{A}$ is commutative (as in Example \ref{op2}), we can compute the $\tau^{6n}$-coefficient of $ \mathrm{Log} (\tau \Delta_{\rho}^{K_1} )$ as $ \tau^{6n}/3n \in \bar{P}_{6n} = \mathcal{A}/\{ a - \kappa (a)\}$, which is non-trivial. We can also verify that the $6$-fold metabelian Alexander polynomial $\Delta_{ \rho^{\rm meta}_{ 6} }^{\Upsilon }$ is computed as $(1-t^6)^2 $.

\subsection{Figure eight knot}
\label{ExaSSs2}
Next, let us consider the figure eight knot $4_1$. The Wirtinger presentation gives us the following presentation:
$$ \pi_1(S^3 \setminus 4_1) \cong \langle u,v \ | \ wv=uw \ \ \rangle $$
where $w= v^{-1}uv u^{-1}$, which can be also written as
$$ \langle \ x_1, x_2 , \mathfrak{m} \ | \ \mathfrak{m} x_1x_2 \mathfrak{m}^{-1} = x_1 , \ \mathfrak{m} x_2 x_1x_2 \mathfrak{m}^{-1} = x_2 \ \ \rangle, \ \ \ \ \mathrm{where } \ \ \ \mathfrak{m} =u^{-1}, \ x_1=w, \ x_2 =vu^{-1}.$$
According to KnotInfo, this presentation gives a monodoromy of the fibered knot $4_1$. Then, from the definition of $A_{F,W}$, we have
$$A_{F,W}= \tau \left(
\begin{array}{rr} \!\!\!\! 1 &- \rho (x_1) \\
-\rho (x_2 ) & 1+\rho (x_2 x_1) \\
\end{array}
\right) - \left(
\begin{array}{rr}
1 & 0\\
0& 1\\
\end{array}
\right)= \left(
\begin{array}{rr} \tau -1 & -\tau \rho (x_1 ) \\
-\tau \rho (x_2 ) & \tau +\tau \rho (x_2 x_1) -1 \\
\end{array}
\right) . $$
Again, let us notice the equality,
\[ \left(
\begin{array}{rr} 1& 0\\
(\tau-1) \rho (x_2^{-1} )\tau^{-1} & 1 \\
\end{array}
\right) \left(
\begin{array}{rr} 0& 1\\
-1 & 0 \\
\end{array}
\right)
A_{F,W} \left(
\begin{array}{rr} 1& -\rho (x_2^{-1} )\tau^{-1} \rho (x_1) \tau \\
0 & 1 \\
\end{array}
\right)\]
\[= \left(
\begin{array}{rr} -\tau \rho( x_2 )& 0 \\
0 & -(\tau - 1)\rho(x_2^{-1})( 1 +\rho( x_2 x_1) -\tau^{-1} ) + \tau \rho(x_1) \\
\end{array}
\right) . \]
Then, $\Delta_{\rho}^{K_1}$ in $\mathcal{Q}_{\mathcal{A}, \kappa}^{K_1}$ is computed as
\[ \Delta_{\rho}^{K_1} = [- A_{F,W} \tau^{-1 }]= (\tau - 1)\rho(x_2^{-1}) ( 1 +\rho( x_2 x_1) -\tau^{-1} ) -\tau \rho(x_1) \]
\[ = \tau^{-1 }
-\rho(x_1^{-1}x_2^{-1})-1-\rho(x_2^2 x_1x_2^{-1} ) +\rho(x_1^{-1}x_2^{-1})\tau \in \mathcal{Q}_{\mathcal{A}, \kappa}^{K_1} . \]
In this case, $\Delta_{\rho}^{K_1}$ formally depends on $\rho$ and $\tau \Delta_{\rho}^{K_1} \in W (\mathcal{A}, \kappa)$. The non-triviality can be shown as follows:
\begin{exa}\label{non.exa}
We will compute the $N$-fold metabelian Alexander polynomial with $N=2,3$. First, let $N=2$. Then, $ \mathrm{Tor}H_1(E_K^2;\Z) \cong \Z /5 \cong \langle x_1, x_2 | x_1^5, x_1^3 x_2 \rangle $ and $\kappa(x_1^n)=x_1^{-n}$. By the above computation, the Alexander polynomial $\Delta_{ \rho^{\rm meta}_{ 2} } $ is equal to
$$\Delta_{ \rho^{\rm meta}_{ 2} }^{K_1} = \tau^{-1} - ( x_1+ 1+ x^4_1) + x_1 \tau . $$
Then, for $ m\leq 3$, the $\tau^{2m}$-coefficients of $\mathrm{Log}(\tau \Delta_{ \rho^{\rm meta}_{ 2} }^{K_1} )$ in $\Q[x_1] /\{x_1^5, x- \kappa (x) \} =\Q[x_1]/(x_1^5) / \{ x_1=x_1^4, x_1^2 =x_1^3\} $ are as follows:
\[\mathrm{Log}_2(\tau \Delta_{ \rho^{\rm meta}_{ 2} }^{K_1} )= (3 + 2 x + x^2 + x^3 + 3 x^4 )/2,\]
\[\mathrm{Log}_4(\tau \Delta_{ \rho^{\rm meta}_{ 2} }^{K_1} )= (21 + 18 x + 15 x^2 + 17 x^3 + 20 x^4)/4 ,\]
\[\mathrm{Log}_6(\tau \Delta_{ \rho^{\rm meta}_{ 2} }^{K_1} )= (171 + 163 x + 157 x^2 + 160 x^3 + 169 x^4 )/6.\]
Moreover, we can verify that the metabelian polynomial $\Delta_{\rho^{\rm meta}_{ 2} }^{K_1} $ is equal to 
$$ t^{-2} - 3 - x_1 - x_1^2 - x_1^3 - x_1^4 + t^2\in \Q[x_1]/(x_1^5) [t,t^{-1}]. $$

Next, let $N=3$. Then, $ \mathrm{Tor}H_1(E_K^3;\Z) \cong \Z/4 x \oplus \Z/4 y$ and $\kappa(x^ay^b )= x^{2a-b}y^{-a+b}$. Similarly, we can obtain 
\begin{eqnarray}
\mathrm{Log}_3(\tau \Delta_{ \rho^{\rm meta}_{ 3} }^{K_1} )= & (3 + 3 x + 2 x^2 + 2 x^3 + y + 2 x y + 3 x^2 y + x^3 y + 2 y^2 +\notag \\
& x y^2 + 2 x^2 y^2 + 2 x^3 y^2 + 2 y^3 + x y^3 + 3 x^2 y^3 + 3 x^3 y^3)/3 , \notag \label{oo9}
\end{eqnarray}
\begin{eqnarray}
\mathrm{Log}_6(\tau \Delta_{ \rho^{\rm meta}_{ 3 } }^{K_1} )=& (81 + 77 x + 71 x^2 + 73 x^3 + 68 y + 74 x y + 77 x^2 y + 71 x^3 y +74 y^2 + \notag \\
& 71 x y^2 +72 x^2 y^2 + 76 x^3 y^2 + 77 y^3 + 76 x y^3 + 75 x^2 y^3 + 76 x^3 y^3)/6. \notag
\end{eqnarray}
Furthermore, the metabelian polynomial $\Delta_{\rho^{\rm meta}_{ 3} }^{\Upsilon} $ is given by
$$ 1 + 4 t^3 + 7 t^3 x y + 3 t^3 x^2 y^2 + 4 t^3 x^3 y^3 + t^6 x y . $$
While computing $ \mathrm{Log}_N(\tau \Delta_{ \rho^{\rm meta}_{ /p^m} }^{\Upsilon} )$ with $ \kappa^N = \mathrm{id}_{\mathcal{A}}$ for other $N$ is feasible, the resulting computations are long, so they have been omitted.
\end{exa}

\subsection{The $5_2$-knot}
\label{ex52}
Next, let us consider the $5_2$-knot, which is not fibered. In private discussions, H. Goda stated that from the algorithm in \cite{GS}, the knot group can be presented as
$$ \langle \ x_1, x_2 , \mathfrak{m} \ | \ \mathfrak{m} x_1^{-2} \mathfrak{m}^{-1} = x_2 x_1^{-2} , \ \mathfrak{m} x_1^{-1} x_2 \mathfrak{m}^{-1} =   x_2 \ \ \rangle.$$
Likewise, from the definition of $A_{F,W}$, we have the following:
$$A_{F,W}= \left(
\begin{array}{rr} \tau \rho (-x_1^{-1}- x_1^{-2}) +\rho( x_2x_1^{-1}+ x_2x_1^{-2})& 1 \ \ \ \\
- \tau \rho (x_1 ^{-1}) \ \ \ \ \ \ \ \ \ \ \ \ \ \ \ \ \ \ \ \ \ \ \ \ & \tau- \rho (x_1)\\
\end{array}
\right) . $$
Then, by using elementary transformations, we can compute $\Delta_{\rho}^{K_1}$ as
\[ A_{F,W} \tau^{-1 }=\Bigl(
-\tau \rho(x_1)+ \bigl(\tau \rho (-x_1^{-1}- x_1^{-2}) +\rho(x_2x_1^{-1}+ x_2x_1^{-2}) \bigr)(\tau - \rho(x_1))\Bigl)\tau^{-1 } \]
\begin{exa}\label{non.exa2} 
Similarly to Example \ref{non.exa}, we will compute the $3$-fold metabelian Alexander polynomial. Let us denote the classes of the generator $x_1, x_2 $ by $x, y$, respectively. Notice that $\mathrm{Tor}H_1(E_L^3;\Z) \cong \Z/5 x \oplus \Z/5 y$ and $\kappa(x^ay^b )= x^{a+b } y^{a/2 -b/2}$. The metabelian Alexander polynomial $\Delta_{ \rho^{\rm meta}_{3} }^{K_1} $ can be computed as
$$ \Delta_{ \rho^{\rm meta}_{ 3} }^{K_1} = (yx^{-1}+ y)\tau^{-1} - ( x^{-1}y+ x^{-2}y-xy^{-1/2 }+1 +x^{-1}y^{1/2}) +( x^{-1}y^{1/2}+x^{-2} y ) \tau . $$
\end{exa}

\section{The proofs of theorems}
\label{appeSS}
We first give the proof of Theorem \ref{thm11} (the invariance of the $K_1$-class $\Delta_{\rho}^{K_1}$), which is outlined in the proof of the main theorem \cite{Lin}. 
While his idea is pioneering, there are mistakes and gaps in his proof, such as in the chain rule of the Fox derivative and many of the derivatives variables. For this reason, we will give detailed proofs. 

\subsection{Invariance from the choice of the generator $x_1,\dots, x_{2g}$}
\label{InvHandle2}
We will show the invariance with respect to the choice of the generator $x_1,\dots, x_{2g}$. We select another basis for the free group $\pi_1(S^3 \setminus F)$, e.g., $x_1', \dots, x_{2g}'$. We use $ \Delta_{\rho}'$ to denote the associated $K_1$-class defined from $x_1',\dots, x_{2g}'$.

In the computations below, $\rho (x)$ is denoted by $x^{\rho}$.

Let us review the chain rule for the Fox derivative and a result of Birman \cite{Bir}. The chain rule \cite[(2.6)]{Fox} gives us the following:
$$ \frac{\partial y_j }{ \partial x_i' } = \sum_{k: 1 \leq k \leq 2g } \Bigl( \frac{\partial y_j}{ \partial x_k }\Bigr) \Bigl( \frac{\partial x_k}{ \partial x_i'} \Bigr), \ \ \ \ \ \ \ \ \ \frac{\partial z_j }{ \partial x_i' }= \sum_{k: 1 \leq k \leq 2g } \Bigl( \frac{\partial z_j}{ \partial x_k }\Bigr) \Bigl( \frac{\partial x_k}{ \partial x_i'} \Bigr).$$
Since the correspondence $\lambda : \{ x_1, \dots, x_{2g}\} \mapsto \{ x_1', \dots, x_{2g}'\}$ is a base change, by the implicit function theorem \cite{Bir}, the Jacobi matrix of entry $\{ \frac{\partial x_k'}{ \partial x_i } \}_{ 1 \leq k,i \leq 2g} $ admits an inverse matrix, which is of the form $\{ \frac{\partial x_k}{ \partial x_i' } \}_{ 1 \leq k,i \leq 2g} $, by using the inverse mapping $\lambda^{-1}$.

The following shows that $ \Delta_{\rho } = \Delta_{\rho}'$ in $\mathcal{Q}_{\mathcal{A}, \kappa} $.
\begin{eqnarray}
\Delta_{\rho}'&=& \tau^{-g} \Bigl\{ \tau \Bigl( \frac{\partial y_j }{\partial x_i '} \Bigr)^{\rho}-\Bigl( \frac{\partial z_j }{\partial x_i '} \Bigr)^{\rho} \Bigr\}_{ 1 \leq i,j \leq 2g} \notag \\
&= & \tau^{-g} \Bigl\{ \displaystyle{ \sum_{1 \leq k \leq 2g}} \tau \Bigl( \frac{\partial y_j}{ \partial x_k }\Bigr)^{\rho } \Bigl( \frac{\partial x_k}{ \partial x_i'} \Bigr)^{\rho }-\Bigl( \frac{\partial z_j}{ \partial x_k }\Bigr)^{\rho } \Bigl( \frac{\partial x_k}{ \partial x_i'} \Bigr)^{\rho } \Bigr\}_{ 1 \leq i,j \leq 2g} \notag \\
&= & \tau^{-g} \Bigl\{ \tau \Bigl( \frac{\partial y_j }{\partial x_i } \Bigr)^{\rho}-\Bigl( \frac{\partial z_j }{\partial x_i } \Bigr)^{\rho} \Bigr\}_{ 1 \leq i,j \leq 2g} \cdot \Bigl\{ \Bigl( \frac{\partial x_j }{\partial x_i' } \Bigr)^{\rho} \Bigr\}_{ 1 \leq i,j \leq 2g} \notag \\
&= &\Delta_{\rho}^{K_1} \cdot \Bigl\{ \Bigl( \frac{\partial x_j }{\partial x_i' } \Bigr)^{\rho} \Bigr\}_{ 1 \leq i,j \leq 2g} \in K_1( \mathcal{A}_{\kappa}( \!( \tau) \!)) \notag \label{oo1}
\end{eqnarray}
Notice that the last matrix is invertible and lies in $K_1( \mathcal{A}) $. Thus $\Delta_{\rho}' $ is equal to $\Delta_{\rho}^{K_1} $ in $ \mathcal{Q}_{\mathcal{A},\kappa}$.

\subsection{Invariance from the choice of spine $W$}
\label{InvHandle4}
We show the independence of the choice of spine $W$. Let us consider another spine $W$, i.e., choose a basis of $\pi_1( F)$, such as $u_1', \dots, u_{2g}'$. The associated $K_1$-class is denoted by $ \Delta_{\rho}'$.

In a similar way to \S \ref{ganeaS}, we have words $ y_i'$ and $z_i'$ of $ x_1, \dots, x_{2g}$, and the group isomorphism:
\begin{equation}\label{jjj2} \pi_1(S^3 \setminus K) \cong \langle x_1, \dots , x_{2g}, \mathfrak{m} \ | \ \mathfrak{m} y_i' \mathfrak{m}^{-1}= z_i' \ \ \ \ i \in \{ 1, \dots , 2g\} \rangle .\end{equation}
From the definition of $y_i$, the words $y_1,\dots, y_{2g}$ have independence; that is, the free subgroup generated by $y_1,\dots, y_{2g}$ has a basis $y_1,\dots, y_{2g}$. Hence, we can define the derivative $\partial / \partial y_i$ as well as the derivative $\partial / \partial z_i$. Furthermore, from the definitions of $ u_i^{\sharp}, u_i^{\flat }$ (see Figure \ref{ftft}), the correspondences
$$ \lambda_y: \{ y_1,\dots, y_{2g}\} \lra \{ y_1',\dots, y_{2g}'\} , \ \ \ \ \lambda_z: \{ z_1,\dots, z_{2g}\} \longmapsto \{ z_1',\dots, z_{2g}'\} $$
are the same. By letting $ F_{yz}:\{ y_1,\dots, y_{2g}\} \ra \{ z_1,\dots, z_{2g}\} $ be the group isomorphism that sends $ y_i$ to $z_i$, we can obtain $F_{yz}' \circ \lambda_y= \lambda_z \circ F_{yz}$. Therefore, by the definition of the Fox derivatives, we have
$ \displaystyle{\frac{\partial z_j' }{\partial z_i} = F_{yz} \Bigl( \frac{\partial y_j' }{\partial y_i} \Bigr) } $ for $1 \leq i,j \leq 2g. $ Since
$\tau \rho (y_i) \tau^{-1} =\rho (z_i) $ by the presentation \eqref{jjj}, we have
\begin{equation}\label{jjj3} \tau \Bigl( \frac{\partial y_j' }{ \partial y_i } \Bigr)^{\rho } \tau^{-1} =\Bigl( \frac{\partial z_j' }{ \partial z_i } \Bigr)^{\rho } \in \mathcal{A}_{\kappa}( \!( \tau) \!), \ \ \ \ \ \mathrm{for} \ \ \ 1 \leq i,j \leq 2g.\end{equation}
Furthermore, we can observe the following from the chain rule:
\begin{equation}\label{jjj4} \Bigl( \frac{\partial y_j' }{ \partial x_i} \Bigr)^{\rho } = \sum_{k: 1 \leq k \leq 2g} \Bigl( \frac{\partial y_j'}{ \partial y_k }\Bigr)^{\rho } \Bigl( \frac{\partial y_k}{ \partial x_i} \Bigr)^{\rho }, \ \ \ \ \ \ \ \ \ \Bigl( \frac{\partial z_j '}{ \partial x_i } \Bigr)^{\rho } = \sum_{k: 1 \leq k \leq 2g} \Bigl( \frac{\partial z_j '}{ \partial z_k }\Bigr)^{\rho } \Bigl( \frac{\partial z_k}{ \partial x_i } \Bigr)^{\rho } .\end{equation}
To summarize, we can now compute $\Delta_{\rho}'$ in $ K_1( \mathcal{A}_{\kappa}( \!( \tau) \!))$ as
\begin{eqnarray}
& \tau^{-g} \Bigl\{ \tau \Bigl( \frac{\partial y_j '}{\partial x_i} \Bigr)^{\rho}-\Bigl( \frac{\partial z_j '}{\partial x_i } \Bigr)^{\rho} \Bigr\}_{ 1 \leq i,j \leq 2g} \notag \\
= & \tau^{-g} \Bigl\{ \displaystyle{ \sum_{1 \leq k \leq 2g}} \tau \Bigl( \frac{\partial y_j' }{ \partial y_k }\Bigr)^{\rho } \Bigl( \frac{\partial y_k}{ \partial x_i} \Bigr)^{\rho }-\Bigl( \frac{\partial z_j'}{ \partial z_k }\Bigr)^{\rho } \Bigl( \frac{\partial z_k}{ \partial x_i} \Bigr)^{\rho } \Bigr\}_{ 1 \leq i,j \leq 2g} \label{oo7} \\
= & \tau^{-g} \Bigl\{ \displaystyle{ \sum_{1 \leq k \leq 2g}} \tau \Bigl( \frac{\partial y_j' }{ \partial y_k }\Bigr)^{\rho } \Bigl( \frac{\partial y_k}{ \partial x_i} \Bigr)^{\rho }- \tau \Bigl( \frac{\partial y_j' }{ \partial y_k } \Bigr)^{\rho } \tau^{-1} \Bigl( \frac{\partial z_k}{ \partial x_i} \Bigr)^{\rho } \Bigr\}_{ 1 \leq i,j \leq 2g} \label{oo8} \\
= & \tau^{-g} \Bigl\{ \tau \Bigl( \frac{\partial y_j' }{\partial y_i } \Bigr)^{\rho} \tau^{-1} \Bigr\}_{ 1 \leq i,j \leq 2g} \cdot \Bigl\{ \tau \Bigl( \frac{\partial y_j }{\partial x_i } \Bigr)^{\rho}-\Bigl( \frac{\partial z_j }{\partial x_i } \Bigr)^{\rho} \Bigr\}_{ 1 \leq i,j \leq 2g} \notag \\
= & \Bigl\{ \tau \Bigl( \frac{\partial y_j ' }{\partial y_i } \Bigr)^{\rho} \tau^{-1} \Bigr\}_{ 1 \leq i,j \leq 2g} \cdot \Delta_{\rho}^{K_1} \in K_1( \mathcal{A}_{\kappa}( \!( \tau) \!)) , \notag \label{oo9}
\end{eqnarray}
where \eqref{oo7} follows from \eqref{jjj4} and \eqref{oo8} follows from \eqref{jjj3}. Hence $\Delta_{\rho}^{K_1} =\Delta_{\rho}'$ in $ \mathcal{Q}_{\mathcal{A}, \kappa} $ as required.

\subsection{Invariance from the Seifert surface}
\label{InvHandle}
Finally, we will show the invariance from the Seifert surface. First, let us recall Theorem 1.7 of \cite{Lin} (cf. \cite[Chapter 8]{Lic}). Let $F$ be a regular Seifert surface with a spine $W$. Let $\alpha $ be an oriented arc in $S^3$ such that $\alpha \cap F = \partial \alpha $ and the intersection is composed of the orientations of $S$ and $\alpha $. If we add a tube to $F$ along $\alpha$, we have another Seifert surface $F'$. This construction of $F'$ obtained from $F$ and $\alpha$ is called a {\it handle addition}. A handle addition to $F$ is said to be {\it regular} if $W \cup \alpha \subset S^3 $ is isotopic to the standard embedding. The inverse construction of a regular handle addition is called a {\it regular handle subtraction}. Let $F$ and $F'$ be two regular Seifert surfaces of the knot $K$. The surfaces are {\it regularly $S$-equivalent} if there exists a sequence of regular Seifert surfaces
$$ F = F_1, F_2, \dots , F_{m-1}, F_m = F'$$
of $K$ such that $F_{i+1}$ is obtained from $F_i$ by either a regular handle subtraction or a regular handle addition. It is shown in \cite[Theorem 1.7]{Lin} that any two regular Seifert surfaces of a knot are regularly $S$-equivalent. Therefore, if we obtain an invariant from a regular Seifert surface, which is invariant with respect to regularly $S$-equivalence, then it is a knot invariant of $K$.

\begin{figure}[htpb]
\begin{center}
\begin{picture}(-60,60)
\put(-191,2){\pc{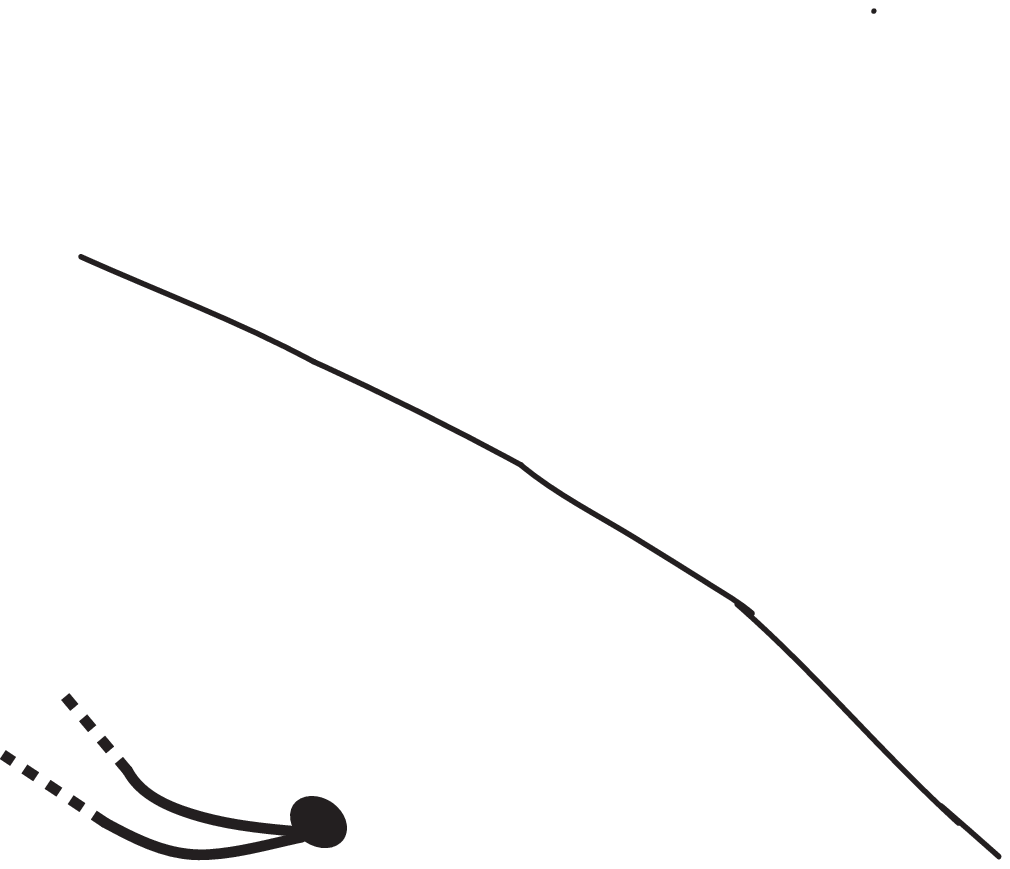}{0.33104}}
\put(-158,38){\large $F $}
\put(-8,38){\large $F' $}
\put(4,15){\large $a_{g+1} $}
\put(45,7){\large $b_{g+1} $}
\put(-65,23){\large $-\!\!\!\! -\!\!\!\! - \!\!\!\! \lra $}

\put(-1,2){\pc{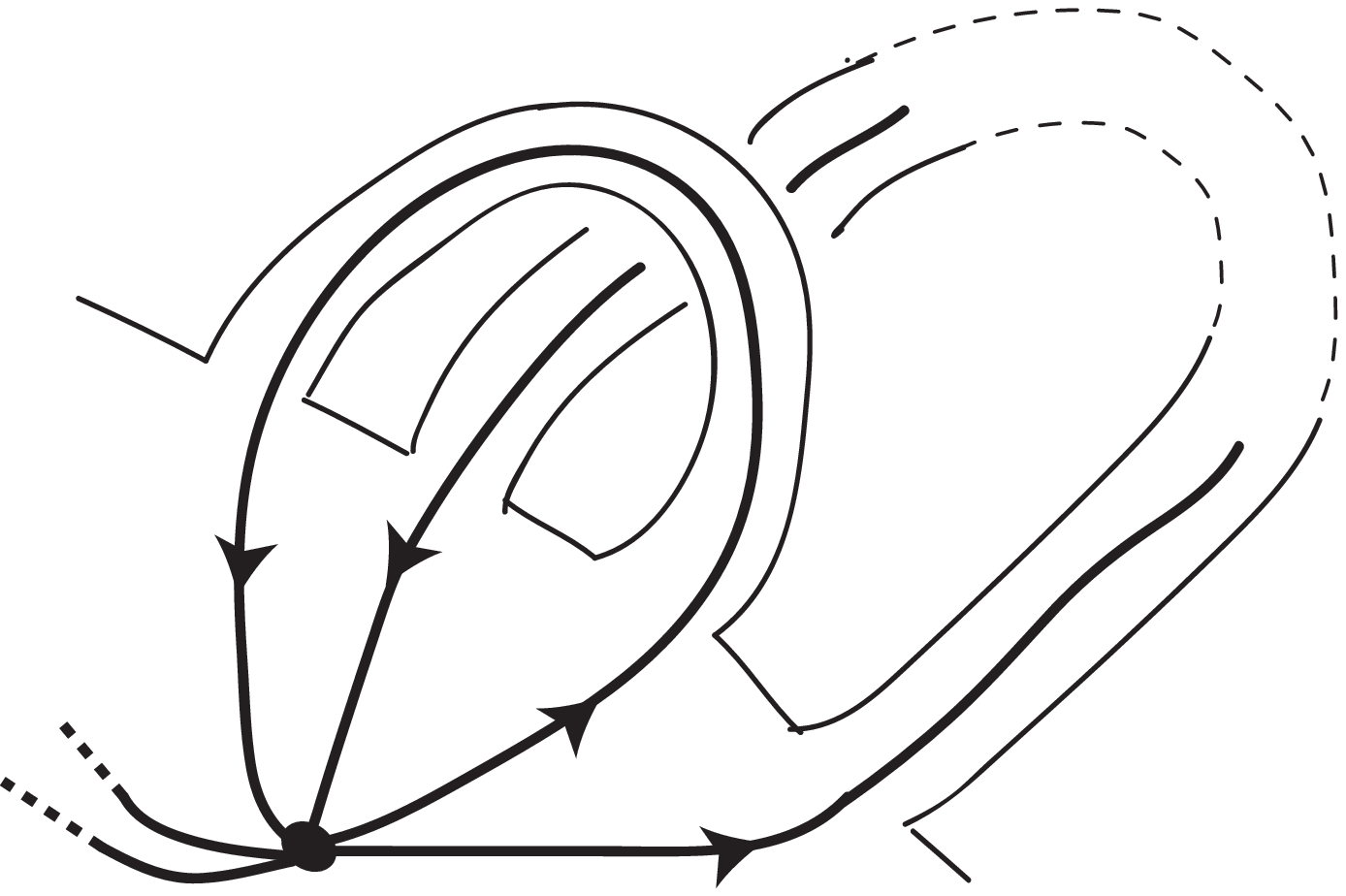}{0.33104}}
\end{picture}
\end{center}
\caption{\label{ftft2} A handle attaching; the resulting surface $F'$ and the loops $a_{g+1}$ and $b_{g+1}$. }
\end{figure}
Next, we will review Lemma \ref{lem12}. Let $F'$ be the new regular Seifert surface obtained from a regular handle addition. Fix circles $a_{g+1}$ and $b_{g+1}$ along the handle (as depicted in Figure \ref{ftft2}) such that the union of the spine $W$, $a_{g+1}$ and $b_{g+1}$ yields a regular spine $\bar{W}$ for $F'$. Use $x_1, \dots, x_{2g+2}$ to denote the associated generator of $\pi_1(S^3 \setminus F')$ with $x_j =x_j$ for $ 1\leq j \leq 2g$. Similarly to \eqref{jjj}, there are words $\bar{y}_1, \dots, \bar{y}_{2g+2}, \bar{z}_{1},\dots, \bar{z}_{2g+2}$ in $ x_1, \dots , x_{2g+2}$ such that $\pi_1(S^3 \setminus K)$ has the presentation
\begin{equation}\label{jjj5} \langle x_1, \dots , x_{2g+2}, \mathfrak{m} \ | \ \mathfrak{m} \bar{y}_i \mathfrak{m} ^{-1}= \bar{z}_i \ \ \ \ i \in \{ 1, \dots , 2g+2\} \ \rangle . \notag \end{equation}
Lin \cite{Lin} showed the properties of the words $\bar{y}_{2g+1}, \bar{y}_{2g+2}, \bar{z}_{2g+1}, \bar{z}_{2g+2}$:
\begin{lem}[{\cite[Lemma 2.4]{Lin}}]\label{lem12} 
For $i = 1, \dots, 2g$, $\bar{y}_i $ and $\bar{z_i}$ are words in $x_1, \dots, x_{2g}, x_{2g+1}$. They are reduced to $ y_i$ and $z_i$ respectively when we set $x_{2g+1} = 1$. Furthermore, $v$ and $w$ are words in $x_1,\dots, x_{2g}$ such that $\bar{y}_{2g+1} = v x_{2g+2}, \bar{y}_{2g+2} = 1, \bar{y}_{2g+1} = w $, and $\bar{z}_{2g+2} = x_{2g+1}$.
\end{lem}
In particular, we have $\Bigl( \frac{\partial \bar{z}_{2g+2} }{ \partial x_{2g+1} } \Bigr)^{\rho}= 1 $ and $\Bigl( \frac{\partial \bar{y}_{2g+1} }{ \partial x_{2g+2} } \Bigr)^{\rho}= \rho (v) .$ Therefore,
\[\Bigl\{ \Bigl( \frac{\partial \bar{y}_j}{ \partial x_i } \Bigr)^{\rho}\Bigr\}_{1 \leq i,j \leq 2g+2}= {\small \left(
\begin{array}{rrr}\Bigl\{\Bigl(
\frac{\partial \bar{y}_j}{ \partial x_i } \Bigr)^{\rho}\Bigr\}_{1 \leq i,j \leq 2g} & * & O \\
* \ \ \ \ \ \ \ \ \ \ & 0 & \rho(v)\\
O \ \ \ \ \ \ \ \ \ \ & 0 &0 \\
\end{array}
\right)}= {\small \left(
\begin{array}{rrr}\Bigl\{\Bigl(
\frac{\partial y_j}{ \partial x_i } \Bigr)^{\rho} \Bigr\}_{1 \leq i,j \leq 2g} & * & O \\
* \ \ \ \ \ \ \ \ \ \ & 0 & \rho(v)\\
O \ \ \ \ \ \ \ \ \ \ & 0 &0 \\
\end{array}
\right)}, \]
\[ \Bigl\{ \Bigl( \frac{\partial \bar{z}_j}{ \partial x_i } \Bigr)^{\rho}\Bigr\}_{1 \leq i,j \leq 2g+2}= {\small \left(
\begin{array}{rrr}\Bigl\{\Bigl(
\frac{\partial \bar{z}_j}{ \partial x_i } \Bigr)^{\rho}_{1 \leq i,j \leq 2g} \Bigr\} & * & O \\
* \ \ \ \ \ \ & 0 & 0\\
O \ \ \ \ \ \ & 1 &0 \\
\end{array}
\right)}= {\small \left(
\begin{array}{rrr}\Bigl\{\Bigl(
\frac{\partial z_j}{ \partial x_i } \Bigr)^{\rho}\Bigr\}_{1 \leq i,j \leq 2g} & * & O \\
* \ \ \ \ \ \ & 0 & 0\\
O \ \ \ \ \ \ & 1 &0 \\
\end{array}
\right)}. \]

To summarize, let us compute $\overline{\Delta_{\rho}}$ in $K_1( \mathcal{A}_{\kappa}( \!( \tau) \!) $ as
\begin{eqnarray}
& \tau^{-g-1} \Bigl( \Bigl\{ \tau \Bigl( \frac{\partial \bar{y}_j}{\partial x_i} \Bigr)^{\rho}-\Bigl( \frac{\partial \bar{z}_j}{\partial x_i } \Bigr)^{\rho} \Bigr\}_{ 1 \leq i,j \leq 2g} \Bigr) \notag \\
= & \tau^{-g-1} \Bigl( \tau {\small \left(
\begin{array}{rrr}\Bigl\{\Bigl(
\frac{\partial z_j}{ \partial x_i } \Bigr)^{\rho}\Bigr\}_{1 \leq i,j \leq 2g} & * & O \\
* \ \ \ \ \ \ \ \ \ \ & 0 & \rho(v)\\
O \ \ \ \ \ \ \ \ \ \ & 0 &0 \\
\end{array}
\right)} - {\small \left(
\begin{array}{rrr}\Bigl\{\Bigl(
\frac{\partial z_j}{ \partial x_i } \Bigr)^{\rho} \Bigr\}_{1 \leq i,j \leq 2g} & * & O \\
* \ \ \ \ \ \ \ \ \ \ \ \ \ \ & 0 & 0\\
O \ \ \ \ \ \ \ \ \ \ \ \ \ \ & 1 &0 \\
\end{array}
\right)} \Bigr) \notag \\
= & \tau^{-g-1} {\small \left(
\begin{array}{rrr} \Bigl\{ \tau
\Bigl(\frac{ \partial y_j}{ \partial x_i } \Bigr)^{\rho} - \Bigl(\frac{\partial z_j}{ \partial x_i } \Bigr)^{\rho} \Bigr\}_{1 \leq i,j \leq 2g} & * & O \\
* \ \ \ \ \ \ \ \ \ \ \ \ \ \ \ \ \ \ & 0 & \tau \rho(v)\\
O \ \ \ \ \ \ \ \ \ \ \ \ \ \ \ \ \ \ & -1 &0 \\
\end{array}
\right)}. \label{oo18}
\end{eqnarray}
Since $-1$ and $\tau \rho(v) $ lay in $ \mathcal{A}^{\times}$, a Gaussian elimination shows that this matrix is equivalent to the following matrix:
$$ \mathcal{D}:= \tau^{-g-1} {\small \left(
\begin{array}{rrr} \Bigl\{ \tau
\Bigl(\frac{ \partial y_j}{ \partial x_i } \Bigr)^{\rho} - \Bigl(\frac{\partial z_j}{ \partial x_i } \Bigr)^{\rho} \Bigr\}_{1 \leq i,j \leq 2g} & O & O \\
O \ \ \ \ \ \ \ \ \ \ \ \ \ \ \ \ \ \ & \tau \rho(v) & 0\\
O \ \ \ \ \ \ \ \ \ \ \ \ \ \ \ \ \ \ & 0 & -1 \\
\end{array}
\right)}. $$
Namely, there exist elementary matrices $A_1 , \dots, A_k$ and $B_1, \dots, B_\ell$ for which $\overline{\Delta_{\rho}} =A_1\cdots A_k \mathcal{D} B_1 \cdots B_\ell $. Since any elementary matrix is zero in the $K_1$-group by Whitehead lemma, we have
$$\overline{\Delta_{\rho}} = - \tau \rho(v) \tau^{-g-1} \Bigl\{ \tau
\Bigl(\frac{ \partial y_j}{ \partial x_i } \Bigr)^{\rho} - \Bigl(\frac{\partial z_j}{ \partial x_i } \Bigr)^{\rho} \Bigr\}_{1 \leq i,j \leq 2g} =\rho(v) \Delta_{\rho}^{K_1} \in K_1( \mathcal{A}_{\kappa}( \!( \tau) \!)). $$
Notice $\rho (v) \in \mathcal{A}^{\times }$ by Assumption $(\dagger)$. Hence $\Delta_{\rho}^{K_1}=\overline{\Delta_{\rho} }$ in $ Q^{\times}_{\rm ab} $ as required.

\subsection{The proofs of Propositions \ref{ww} and \ref{ww3}}
\label{wwS}
Let $ \mathfrak{m}'$ be another meridian. Since any two meridians are conjugate, $h \in \pi_1(S^3 \setminus K)$ such that $ \mathfrak{m}' =h \mathfrak{m}h^{-1} $. Consider other generators $x_i', y_i',$ and $z_i' $ defined by $ h x_i h^{-1}, h y_i h^{-1}$, and $h z_i h^{-1}$, respectively. Then, we have the presentation for $ \pi_1(S^3 \setminus K)$ of the formula:
\begin{equation}\label{jjj5} \langle x_1', \dots , x_{2g}', \mathfrak{m}' \ | \ \mathfrak{m}' y_i' (\mathfrak{m}') ^{-1}= z_i' \ \ \ \ i \in \{ 1, \dots , 2g\} \ \rangle . \notag \end{equation}
As per our definitions, we can observe
$$ \Bigl(\frac{ \partial y_j'}{ \partial x_i' } \Bigr)^{\rho}=
\rho(h) \Bigl(\frac{ \partial y_j}{ \partial x_i } \Bigr)^{\rho} \rho(h^{-1}), \ \ \ \ \ \ \ \ \ \ \Bigl(\frac{\partial z_j '}{ \partial x_i' } \Bigr)^{\rho}= \rho(h ) \Bigl(\frac{\partial z_j}{ \partial x_i } \Bigr)^{\rho} \rho(h^{-1})$$
Take a ring isomorphism $ \mathcal{T} : \mathcal{A}_{\kappa'}( \!( \tau') \!) \ra \mathcal{A}_{\kappa}( \!( \tau) \!)) $ which sends $ a(\tau')^m$ to $ \rho(h^{m }) a \tau^m \rho(h^{-m })$. Hence, we can compute $ \mathcal{T} ( \Delta_{\rho}') $ as
$$ \mathcal{T} ( \Bigl\{ \tau '
\Bigl(\frac{ \partial y_j'}{ \partial x_i '} \Bigr)^{\rho} )- \Bigl(\frac{\partial z_j '}{ \partial x_i ' } \Bigr)^{\rho} \Bigr\}_{1 \leq i,j \leq 2g} ) = \rho (h)
\Bigl\{ \tau
\Bigl(\frac{ \partial y_j}{ \partial x_i } \Bigr)^{\rho} - \Bigl(\frac{\partial z_j}{ \partial x_i } \Bigr)^{\rho} \Bigr\}_{1 \leq i,j \leq 2g} \rho (h^{-1}).
$$
\begin{proof}[Proof of Proposition \ref{ww}]
Denote the $\tau^k$-coefficient of $ \mathrm{Log}$ by $ \mathrm{Log}_k$. By definition, we recognize $\mathrm{Log}'= \mathrm{Log} \circ \mathcal{T} $. After applying the logarithm, we have
$$ \mathrm{Log}_k( (\tau')^g \Delta_{\rho}') \rho (h)^k = \rho (h)^k \mathrm{Log}_k ( \tau^g \Delta_{\rho}) \in P_k.$$
From the definition of $ \bar{P}_k$, we can conclude $ \mathrm{Log}_k( (\tau')^g \Delta_{\rho}') = \mathrm{Log}_k( \tau^g \Delta_{\rho})\in \bar{P}_k $ as required.
\end{proof}
\begin{proof}[Proof of Proposition \ref{ww3}]
By definition, we recognize $ \Upsilon_*' ( \Delta_{\rho}' ) = \Upsilon ( \rho (h) \Delta_{\rho}^{K_1} \rho (h^{-1})) $ as a matrix. Since $\Upsilon$ is a ring homomorphism, $ [\Upsilon_*' ( \Delta_{\rho}' ) ]=[\Upsilon_* ( \Delta_{\rho} ]$ in $K_1(\mathcal{A}_{\rm id}( \!( t ) \!)) )$ as required.
\end{proof}

\subsection{The proofs of Theorems \ref{thm119} and \ref{ww2}}
\label{wwS4}
For the proofs, we first need to establish terminology. For $r \in \mathcal{A}_{\kappa}( \!( \tau) \!)$ and $i \neq j$, $1 \leq i,j \leq n$, we define the elementary matrix, $e_{ij} (r)$, to be the $(n \times n )$-matrix with 1's on the diagonal, $r$ in the $(i,j)$-slot, and $0$'s elsewhere. Recall from \eqref{jjj775588} the multiplicative subgroup $ W(\mathcal{A}, \kappa ) \subset \mathcal{A}_{\kappa}( \!( \tau) \!)^{\times}$. For $ k \leq 2g, $ a matrix $S \in \mathrm{Mat}(2g \times 2g, \mathcal{A}_{\kappa}( \!( \tau) \!))$ is said to satisfy $ (*)_k$ if the following holds:

\

$(*)_k$: for any $i,j$ with $i \neq j$, the $(i,i)$-th entry of $S$ lies in $ W(\mathcal{A}, \kappa ) $ and the $(i,j)$-th entry of $S$ forms $ \tau \mu$ for some $ \mu \in \mathcal{A}_{\kappa} [\![ \tau ]\!] $ . In addition, for any $(i,j)$-th entry of $S$ is zero if $ i < k, j < k, \ i \neq j$.

\

If a matrix $S_k$ satisfies $ (*)_k$ and $a_{ij}$ denotes the $(i,j)$-th entry of $S_k$, let us define $S_{k+1}'$ and $S_{k+1}$ to be
\[ S_{k+1}' := e_{k k+1}(- a_{k k+1} a_{kk}^{-1}) e_{k k+2}(- a_{k k+2} a_{kk}^{-1} ) \cdots e_{k n}(- a_{k n} a_{kn}^{-1})S_k, \]
\[ S_{k+1} := S_{k+1}' e_{k+1, k}( - a_{kk}^{-1} a_{k+1 , k} ) e_{k+2, k}(- a_{kk}^{-1} a_{k+1 , k} ) \cdots e_{n k }(- a_{kk}^{-1} a_{n , k}).\]
Then, we can easily verify that $ S_{k+1}$ satisfies $ (*)_{k+1}$.

\begin{proof}[Proof of Theorem \ref{thm119}]
We first show the ``if" part. Let us regard the fiber as a regular Seifert surface and consider where the situation $S_0$ is $ A_{F,W}$. By the fibered assumption, there is a group isomorphism $\phi: F_{2g} \ra F_{2g}$ such that the word $z_i$ may be $x_i$ and $ y_i=\phi(x_i)$. Since $ \partial z_i / \partial x_j= \delta_{ij}$, $S_0$ satisfies $(*)_0$. Therefore, from the inductive discussion above, we can determine that  $S_{2g}$ is a diagonal matrix satisfying $ (*)_{2g}$. Since every diagonal entry of $S_{2g}$ lies in $W(\mathcal{A}, \kappa ) $, $S_{2g}$ is invertible and $A_{F,W} $ is also invertible as required.

Let us show the converse. Let $G$ and $ f$ in Example \ref{exppp0} be $ \pi_1(S^3 \setminus K)$ and $\mathrm{id}_{G}$, respectively, giving us $\rho: \Z [\pi_1(S^3 \setminus K) ] \ra \mathcal{A}_{\kappa}( \!( \tau) \!)$. By the exact sequence \eqref{jjj5655}, the invertibility of $A_{F,W}$ deduces the vanishing of the associated homology $H_1( S^3 \setminus K; \mathcal{A}_{\kappa}( \!( \tau) \!))$. If $H_1( S^3 \setminus K; \mathcal{A}_{\kappa}( \!( \tau) \!))=0$, then $ K$ is fibered, as shown in \cite[Theorem 1.7]{Fri}, thus completing the proof.
\end{proof}
\begin{proof}[Proof of Theorem \ref{ww2}]
Let $b_{ii}$ denote the $(i,i)$-th entry of $S_{2g} $ in the previous proof. By $(*)_{2g}$, $b_{ii}$ lies in $W(\mathcal{A}, \kappa ) $. In the $K_1$ group, $[A_{F,W} ]$ is diagonal and is equal to $ b_{11} \cdots b_{2g,2g} \in W_1( \mathcal{A})$ as desired.
\end{proof}

\normalsize
DEPARTMENT OF
MATHEMATICS
TOKYO
INSTITUTE OF
TECHNOLOGY
2-12-1
OOKAYAMA
, MEGURO-KU TOKYO
152-8551 JAPAN

\end{document}